\documentclass[12pt]{amsart}

\newtheorem{theorem}{Theorem}[section]
\newtheorem{lemma}[theorem]{Lemma}

\theoremstyle{definition}

\theoremstyle{remark}
\newtheorem{remark}[theorem]{Remark}

\numberwithin{equation}{section}

\usepackage{amsmath}
\usepackage[english]{babel}
\usepackage{babelbib}
\usepackage{natbib}
\RequirePackage{geometry}
\usepackage[utf8]{inputenc}
\usepackage[T1]{fontenc}
\setcitestyle{square,numbers}

\usepackage{comment}
\usepackage{float}
\usepackage{etoolbox}
\usepackage{txfonts}

\usepackage[pdftex]{graphicx}

\usepackage{listingsutf8}

\usepackage{geometry}
\usepackage[usenames,dvipsnames,svgnames,table]{xcolor}
\usepackage[mathcal]{euscript}
\usepackage{mathrsfs}
\usepackage{setspace}
\newcommand{\Zi}{\mathbb{Z}[i]}
\newcommand{\ds}{\displaystyle}

\newcounter{ctProba}
\DeclareMathAlphabet\mathcal{OMS}{cmsy}{m}{n}

\newcommand{\R}{\mathbb{R}}
\newcommand{\Z}{\mathbb{Z}}

\newcommand{\N}{\mathbb{N}}
\newcommand{\Q}{\mathbb{Q}}
\newcommand{\OK}{\mathbb{Z}[i]}%\mathcal{O}_K}

\DeclareMathOperator{\Max}{Max}

\makeatletter
\newtheorem*{rep@theorem}{\rep@title}
\newcommand{\newreptheorem}[2]{%
\newenvironment{rep#1}[1]{%
 \def\rep@title{#2 \ref{##1}}%
 \begin{rep@theorem}}%
 {\end{rep@theorem}}}
\makeatother

\theoremstyle{plain}
\theoremstyle{plain}
\newreptheorem{theorem}{Teorem}

\theoremstyle{plain}
\theoremstyle{plain}
\newtheorem{proposition}[theorem]{Proposition}
\theoremstyle{plain}

\theoremstyle{definition}

\theoremstyle{remark}

\DeclareMathOperator{\Rez}{Re}
\DeclareMathOperator{\Imz}{Im}
\begin{document}
\author{Nikola Ad\v zaga}
\address{Nikola Ad\v zaga, Faculty of Civil Engineering, University of Zagreb}
\email{nadzaga@grad.hr}

\author{Alan Filipin}
\address{Alan Filipin, Faculty of Civil Engineering, University of Zagreb}
\email{filipin@grad.hr}

\author{Zrinka Franušić}
\address{Zrinka Franušić, Faculty of Science, University of Zagreb}
\email{fran@math.hr}

\newgeometry{top=5.5cm, left=2cm, right=1.8cm}
\title{On the extensions of the Diophantine triples in Gaussian integers} %?

%\address{Department of Mathematics, Faculty of Civil Engineering, University of Zagreb, Ka\v ci\' ceva 26, Zagreb, Croatia}

 %11J13? 11G05?

\date{\today}

\begin{abstract}
A Diophantine $m$-tuple is a set of $m$ distinct integers such that the product of any two distinct elements plus one is a perfect square. In this paper we study the extensibility of a Diophantine triple $\{k-1, k+1, 16k^3-4k\}$ in Gaussian integers $\Z{[i]}$ to a Diophantine quadruple. Similar one-parameter family, $\{k-1, k+1, 4k\}$, was studied in \cite{Zrinka},where it was shown that the extension to a Diophantine quadruple is unique (with an element $16k^3-4k$). The family of the triples of the same form $\{k-1, k+1, 16k^3-4k\}$ was studied in rational integers in \cite{16k3z}. It appeared as a special case while solving the extensibility problem of Diophantine pair $\{k-1, k+1\}$, in which it was not possible to use the same method as in the other cases. As authors (Bugeaud, Dujella and Mignotte) point out, the difficulty appears because the gap between $k+1$ and $16k^3-4k$ is not sufficiently large. We find the same difficulty here while trying to use Diophantine approximations. Then we partially solve this problem by using linear forms in logarithms.
\end{abstract}
\maketitle
\setstretch{1.1}

\section{Introduction}
A long-standing conjecture, motivated by work of Baker and Davenport \cite{baker}, that there is no Diophantine quintuple, was proven by He, Togbé and Ziegler \cite{nemapetorke}. In other rings of integers, there are not many results. E.g.~we find only about $10$ papers solving similar problems in the ring of Gaussian integers. We can highlight \cite{BFT} and \cite{Zrinka}, which deal with the extension of Diophantine triples from one-parameter families, and \cite{yo}, which shows that there is no Diophantine $m$-tuple in imaginary quadratic number ring with $m\geqslant 43$.

We deal with a parametric family of triples $\{k-1, k+1, 16k^3-4k\}$, but we start with a general triple and show some results which are useful for any family of triples. Assume that a Diophantine triple $\{a, b, c\}$ in Gaussian integers $\Zi$ can be extended with a fourth element $d$. By eliminating $d$ from the equations it satisfies ($ad+1=x^2$, $bd+1=y^2$ and $cd+1=z^2$), we get a system of two Pell-type equations with common unknown. We show that the structure of the solutions of this system is the same as in the rational integers case.
A solution of this system gives us two simultaneous approximations of square roots close to $1$. One can use Diophantine approximations in the general case (by assuming that $|c|$ is much bigger than $|b|$, say $|c|>|b|^{15}$), which was done in \cite{yo}. However, here we show that this is not useful for the triple of the form $\{k-1, k+1, 16k^3-4k\}$. %Then we apply a variant of Bennett's theorem developed for imaginary quadratic number rings by Jadrijević and Ziegler \cite{Borka}.

We also prove that the linear form in logarithms usually involved in approaching these problems is not zero under certain conditions. This might be useful in lowering the general upper bound, and we also use it here to partially resolve the extensibility problem of the triple $\{k-1, k+1, 16k^3-4k\}$.

%\begin{theorem}
%Let $k$ be a Gaussian integer such that $\Rez k \neq 0$ and $|k|\geqslant 5\cdot 10^{37}$. A Diofantova trojka $\{k-1, k+1, 16k^3-4k\}$ do Diofantove četvorke može proširiti samo s $d=4k$ ili $d=64k^5 -48k^3+8k$. 
%\end{theorem}

\section{System of Pell-type equations}
Let $\{a, b, c\} \subset \OK$ be a Diophantine triple in Gaussian integers $\OK$. Without loss of generality, we may assume $0 < |a| \leqslant |b| \leqslant |c|$. Then there are $r, s$ and $t$ in $\OK$ such that $ab+1 = r^2, ac+1=s^2, bc+1=t^2.$ In \cite{yo}, the following lemma was proven (for general imaginary quadratic number rings).
 
\begin{lemma}\label{acnekva} If $\{a, b, c\}$ is a Diophantine triple in the imaginary quadratic number ring $\OK$ and $abc\neq 0$, then $ab$, $ac$ and $bc$ are not squares in $\Zi$.
\end{lemma}
If there is $d\in \OK$ such that $\{a, b, c, d\}$ is a Diophantine quadruple, then there are $x, y, z\in\OK$ such that $ad+1 = x^2, bd+1=y^2, cd+1=z^2$. Eliminating $d$ implies that
\begin{align}
az^2-cx^2 &= a-c \label{Eq:Pellac}\\
bz^2-cy^2 &= b-c .\label{Eq:Pellbc}
\end{align}
These equations are similar to Pell's equations and their solutions have a very similar structure. The solutions of Pell-type equations ($x^2-Dy^2=N$) in imaginary quadratic rings are described in \cite{fjel}, as well as here, in a slightly different manner, adapted for the problem at hand.

\begin{lemma}\label{lem:fund}
There are positive integers $i_0$ and $j_0$, elements $z_0^{(i)}, x_0^{(i)}, z_1^{(j)}, x_1^{(j)}$ of $\OK$, for $i = 1, \dotsc, i_0$ and $j = 1, \dotsc, j_0$, such that:
\begin{itemize}
\item[a)] $(z_0^{(i)}, x_0^{(i)})$ are solutions of \eqref{Eq:Pellac}, while $(z_1^{(j)}, y_1^{(j)})$ are solutions of \eqref{Eq:Pellbc}. The solutions denoted here are called \emph{fundamental}.
\item[b)] Fundamental solutions satisfy the following inequalities:
\begin{align*}
1 \leqslant |x_0^{(i)}| &\leqslant \sqrt{\frac{|a||c-a|}{|s|-1}}, &
1 \leqslant |z_0^{(i)}| &\leqslant \sqrt{\frac{|c-a|}{|a|} + \frac{|c||c-a|}{|s|-1}},\\
1 \leqslant |y_1^{(j)}| &\leqslant \sqrt{\frac{|b||c-b|}{|t|-1}}, &
1 \leqslant |z_1^{(j)}| &\leqslant \sqrt{\frac{|c-b|}{|b|} + \frac{|c||c-b|}{|t|-1}}. 
\end{align*}
\item[c)] If $(z, x)$ is the solution of \eqref{Eq:Pellac}, then there are $i \in \{1, \dotsc, i_0\}$ and $m\in\Z$ such that \[ z\sqrt{a}+x\sqrt{c} = (z_0^{(i)}\sqrt{a}+x_0^{(i)}\sqrt{c})(s+\sqrt{ac})^m. \]
If $(z, y)$ is the solution of \eqref{Eq:Pellbc}, then there are $j \in \{1, \dotsc, j_0\}$ and $n \in\Z$ such that\[ z\sqrt{b}+y\sqrt{c} = (z_1^{(j)}\sqrt{a}+y_1^{(j)}\sqrt{c})(t+\sqrt{bc})^n.\]
\end{itemize}
\end{lemma}

\begin{proof}
If $(x, z)$ is the solution of \eqref{Eq:Pellac}, then the pairs $(x_m, y_m) \in \OK^2$, defined as \begin{equation}
 x_m\sqrt{c}+z_m\sqrt{a} = (x\sqrt{c}+z\sqrt{a})(s+\sqrt{ac})^m\label{eq:pellm}
\end{equation} are also the solutions of \eqref{Eq:Pellac} for every $m \in \Z$. We prove this inductively: for $m=1$, we have $x_1\sqrt{c}+z_1\sqrt{a} = (x\sqrt{c}+z\sqrt{a})(s+\sqrt{ac}) =(sx+az)\sqrt{c}+(sz+cx)\sqrt{a}$, i.~e.~$x_1 = sx+az, z_1 = sz+cx$. Let us note here that we have used Lemma \ref{acnekva}. Then
\begin{align*}
az_1^2-cx_1^2 &= a(sz+cx)^2-c(sx+az)^2 = s^2(az^2-cx^2)+ac(cx^2-az^2)\\
				&= s^2(a-c)+ac(c-a) = (s^2-ac)(a-c)=a-c.
\end{align*}
Inductively it follows that $(x_m, z_m)$ is the solution of \eqref{Eq:Pellac} for every $m\in\N_0$. Analogously one resolves the case $m=-1$ to conclude that $(x_m, z_m)$ is the solution \eqref{Eq:Pellac} for every $m\in\Z$.

Let $(x^*, z^*)$ be the solution such that $|x^*|$ is minimal among the solutions from the sequence $(x_m, z_m)_{m\in\Z}$ defined in \eqref{eq:pellm}.
Let us denote the next and the previous solution in the sequence by $(x', z')$ and $(x'', z'')$. More precisely, let $x' = sx^*+az^*, z' = sz^*+cx^*$ and $x''=sx^*-az^*, z''=sz^*-cx^*$. Then $|x'| \geqslant |x^*|$ and $|x''|\geqslant |x^*|$. On the other hand, by $|x'|+|x''| \geqslant |x'+x''|=2|s||x^*|$ it follows that $|x'| \geqslant |sx^*|$ or $|x''| \geqslant |sx^*|$. In any case, we can conclude that the product $|x'x''| \geqslant |sx^*| \cdot |x^*|$. Since $(z^*, x^*)$ is the solution of \eqref{Eq:Pellac}, we obtain equivalent inequalities $|(sx^*+az^*)(sx^*-az^*)| \geqslant |s| \cdot |x^*|^2$, $|(ac+1)(x^*)^2-a^2(z^*)^2| \geqslant |s| \cdot |x^*|^2$ and $|a(c-a)+(x^*)^2| \geqslant |s| \cdot |x^*|^2$.

%Tada je $|z'| \geqslant |z^*|$ i $|z''|\geqslant |z^*|$. S druge strane, kako je $|z'|+|z''| \geqslant |z'+z''|=2|s||z^*|$, slijedi da je $|z'| \geqslant |sz^*|$ ili $|z''| \geqslant |sz^*|$. U svakom slučaju, vrijedi da je produkt $|z'z''| \geqslant |sz^*| \cdot |z^*|$. Dobivamo niz ekvivalentnih nejednakosti $|z'z''| \geqslant |s| \cdot |z^*|^2$, tj.~$|(sz^*+cx^*)(sz^*-cx^*)| \geqslant |s| \cdot |z^*|^2$, $|s^2(z^*)^2-c^2(x^*)^2| \geqslant |s| \cdot |z^*|^2$ i na kraju $|c(a-c)+(z^*)^2 \geqslant |s| \cdot |z^*|^2$, jer je $s^2=ac+1$, a $(z^*, x^*)$ je rješenje \eqref{Eq:Pellac}.

%Iz posljednje dobivene nejednakosti izvodimo gornju granicu za $|z^*|$:\\
%$|c|\cdot |c-a|+|z^*|^2 \geqslant |c(a-c)+(z^*)^2 \geqslant |s| \cdot |z^*|^2$, pa je $|c|\cdot |c-a| \geqslant (|s|-1)|z^*|^2$ i konačno $\ds |z^*|^2 \leqslant \frac{|c|\cdot|c-a|}{|s|-1}$.

We derive the upper bound on $|x^*|$ from the last inequality, $|a|\cdot |c-a|+|x^*|^2 \geqslant |a(c-a)+(x^*)^2| \geqslant |s| \cdot |x^*|^2$, so $|a|\cdot |c-a| \geqslant (|s|-1)|x^*|^2$, and finally
$\ds |x^*|^2 \leqslant \frac{|a|\cdot|c-a|}{|s|-1}$.

This bound on $|x^*|$ implies an upper bound on $|z^*|$,
$\ds |z^*| = \frac{|c(x^*)^2-c+a|}{|a|} \leqslant \frac{|c||x^*|^2}{|a|}+\frac{|c-a|}{|a|}$. Without loss of generality, we may assume that $x_0 = x^*, z_0 = z^*$.

Analogously one gets the upper bounds on fundamental solutions of the equation \eqref{Eq:Pellbc}.
\end{proof}

From $(c)$ part of the Lemma \ref{lem:fund} one can obtain and solve the same recurrence relations as in the integer case (see \cite{duje}). More precisely, the following lemma holds
\begin{lemma}\label{lema:rekurzije}
Every solution $z$ of the equation \eqref{Eq:Pellac} is contained in one of the following sequences
\begin{equation} \label{recv}
v_0^{(i)} = z_0^{(i)}, \quad v_1^{(i)} = sz_0^{(i)}+cx_0^{(i)}, \quad v_{m+2}
^{(i,)} = 2sv_{m+1}^{(i)}-v_m^{(i)} \quad \text{ for } i=1, \dotsc, i_0.
\end{equation}

Similarly, every solution $z$ of the equation \eqref{Eq:Pellbc} is contained in one of the following sequences
\begin{equation} \label{recw}
w_0^{(j)} = z_1^{(j)}, \quad w_1^{(j)} = tz_1^{(j)}+cy_1^{(j)}, \quad w_{n+2}^{(i)} = 2tw_{n+1}^{(j)}-w_n^{(j)}, \quad \text{ for } j=1, \dotsc, j_0.
\end{equation}
\end{lemma}
\begin{comment}
\begin{proof}Rekurzivna relacija \eqref{recv} slijedi iz \eqref{eq:pellm}. Prema njoj je 
\begin{align*}
x_{m+1}\sqrt{c}+z_{m+1}\sqrt{a} &= (x\sqrt{c}+z\sqrt{a})(s+\sqrt{ac})^{m+1}\\
&= (x\sqrt{c}+z\sqrt{a})(s+\sqrt{ac})^{m} (s+\sqrt{ac})\\
&= (x_{m}\sqrt{c}+z_{m}\sqrt{a})(s+\sqrt{ac})\\
&= (sx_m+az_m)\sqrt{c}+(sz_m+cx_m)\sqrt{a},
\end{align*}
pa je $x_{m+1}=sx_m+az_m$ i $z_{m+1}=sz_m+cx_m$. Iz toga je
\begin{align*}
z_{m+2} &= sz_{m+1}+cx_{m+1}\\
&= sz_{m+1}+c(sx_m+az_m)\\
&= sz_{m+1}+acz_m+scx_m\\
&= sz_{m+1}+acz_m+s(z_{m+1}-sz_m)\\
&= 2sz_{m+1}+(ac-s^2)z_m\\
&= 2sz_{m+1}-z_m.
\end{align*}
Dakle, za sve nizove rješenja $z_m^{(i)}$ vrijedi rekurzivna relacija istog oblika \eqref{recv}. Kako ne bismo označavali rješenja dviju različitih jednadžbi istim oznakama, označili smo sa $v_m^{(i)}$ nizove rješenja $z$ jednadžbe \eqref{Eq:Pellac}. Nizove rješenja $z$ jednadžbe \eqref{Eq:Pellbc} označili smo s $w_n^{(j)}$ i analogno se dokazuje njihova rekurzivna relacija \eqref{recw}.
\end{proof}
\end{comment}
\noindent We skip the proof as it is the same as in the case of rational integers (see \cite{duje}).

If $d$ extends the initial triple $\{a, b, c\}$, then $z$ is the solution of both equations \eqref{Eq:Pellac} and \eqref{Eq:Pellbc}. Such $z$ is contained in one of the sequences $v_m^{(i)}$ and in one $w_n^{(j)}$, i.~e.\ $z = v_m^{(i)} = w_n^{(j)}$.

By solving the recurrences \eqref{recv} and \eqref{recw}, we obtain
\begin{align}
v_m^{(i)} &= \frac{1}{2\sqrt{a}} \left( (z_0^{(i)} \sqrt{a}+x_0^{(i)}\sqrt{c})(s+\sqrt{ac})^m + (z_0^{(i)} \sqrt{a}-x_0^{(i)}\sqrt{c})(s-\sqrt{ac})^m \right) \\
w_n^{(j)} &= \frac{1}{2\sqrt{b}} \left( (z_1^{(j)} \sqrt{b}+y_1^{(j)}\sqrt{c})(t+\sqrt{bc})^n + (z_1^{(j)} \sqrt{b}-y_1^{(j)}\sqrt{c})(t-\sqrt{bc})^n \right).
\end{align}

Let us denote $P = \frac{1}{\sqrt{a}}(z_0^{(i)}  \sqrt{a}+x_0^{(i)} \sqrt{c})(s+\sqrt{ac})^m, \text{ and }  Q= \frac{1}{\sqrt{b}}(z_1^{(j)} \sqrt{b}+y_1^{(j)}\sqrt{c})(t+\sqrt{bc})^n$.
By $z = v_m^{(i)} = w_n^{(j)}$, it follows that $P-\frac{c-a}{a}P^{-1} = Q-\frac{c-b}{b}Q^{-1}$.
 \begin{lemma}If $|c| \geqslant 4|b|$ and $m, n \geqslant 3$, then $\displaystyle |P| > 12\left|\frac{c}{a}\right|$ and $\displaystyle |Q| > 12\left|\frac{c}{b}\right|$.
 \end{lemma}
 \begin{proof}
Observe that $|s+\sqrt{ac}| \geqslant \sqrt{|ac|}$, since
$\displaystyle \text{Re}\sqrt{1+\frac{1}{ac}} > 0 \Rightarrow \left|\sqrt{1+\frac{1}{ac}}+1\right| > 1 \Rightarrow |\sqrt{ac+1}+\sqrt{ac}| > \sqrt{|ac|}$.
Therefore
\begin{align*}
|P| &= \frac{1}{\sqrt{|a|}} |z_0 \sqrt{a} + x_0\sqrt{c}|\cdot |s+\sqrt{ac}|^m \geqslant \frac{|c-a|}{\sqrt{|a|}} \frac{1}{|z_0\sqrt{a}-x_0\sqrt{c}|} \sqrt{|ac|}^m \\
&\geqslant \frac{|c|-|a|}{\sqrt{|a|}} \frac{1}{|x_0\sqrt{c}|+|z_0|\sqrt{|a|}} |ac|^{3/2} \geqslant \frac{|c|-|c|/4}{\sqrt{|a|}} \frac{1}{3|c|/2 + |c|\sqrt{|a|}} |ac|^{3/2}
\end{align*}
Hence $\ds P\geqslant \frac{3|a||c|^{3/2}}{6+4\sqrt{|a|}} > 12\left|\frac{c}{a}\right|$. The last inequality is equivalent with $|a|^2 |c|^{1/2} > 2(12+8\sqrt{|a|})$. Since $|c|>4|a|$, it suffices to show that $2|a|^{5/2} >2(12+8|a|^{1/2})$,which holds for $|a|^{1/2} \geqslant 2$.

Analogously, $|c| \geqslant 4|b|$ and $n\geqslant 3$ imply $\displaystyle |Q| > 12\left|\frac{c}{b}\right|$.
\end{proof}
Therefore
\begin{align*}
\left| |P|-|Q| \right| &\leqslant |P-Q| = \left| \frac{c-a}{a}P^{-1} -\frac{c-b}{b}Q^{-1}  \right|  \leqslant \Big| \frac{c}{a}-1 \Big| \frac{1}{|P|} + \left| \frac{c}{b}-1 \right| \frac{1}{|Q|} \\
& \leqslant \left| \frac{c}{a}-1 \right| \cdot \frac{1}{12} \left|\frac{a}{c}\right| +  \left| \frac{c}{b}-1 \right| \cdot \frac{1}{12} \left|\frac{b}{c}\right|  \leqslant \frac{1}{12}\left( \Big|1-\frac{a}{c}\Big|+\Big|1-\frac{b}{c}\Big|\right) \leqslant \frac{5}{24}.
%& \leqslant \frac{1}{12}\left(2+\frac 14+\frac 14\right)\\
\end{align*}

\noindent It follows that $\displaystyle \Bigg| \frac{|P|-|Q|}{|P|}\Bigg| \leqslant \frac{5}{24} |P|^{-1} \leqslant \frac{5}{24} < 1$. We now apply the following simple Lemma B.2 from \cite{Smart}.

\begin{lemma} Let $\Delta > 0$ such that $|\Delta-1|\leqslant a$. Then
\[ |\log\Delta| \leqslant \frac{-\log{(1-a)}}{a}|\Delta-1|.\]
\end{lemma}

\noindent We obtain the following inequalities for $\Lambda=\frac{|Q|}{|P|}$:
\begin{align*}
|\log\Lambda|  &= \Bigg|\log{\frac{|Q|}{|P|}}\Bigg|\leqslant \frac{24}{5} \log{\frac{24}{19}} \cdot \frac{5}{24} |P|^{-1} \leqslant \log{\frac{24}{19}} |P|^{-1}  \leqslant \log{\frac{24}{19}} \sqrt{|a|} \frac{|x_0|\sqrt{|c|}+|z_0|\sqrt{|a|}}{|c-a|} |s+\sqrt{ac}|^{-m} \\ & \leqslant \log{\frac{24}{19}} \sqrt{|a|} \frac{3|c|^{3/2} /2 +|c|^{3/2}/2}{3|c|/4}|s+\sqrt{ac}|^{-m} =  \frac 83 \log{\frac{24}{19}} \sqrt{|ac|} |s+\sqrt{ac}|^{-m}.
\end{align*}

\begin{lemma}\label{lforma} If $K=\frac{8}{3}\log{\frac{24}{19}}$, then
\[ |\log\Lambda| = \left|m\log|s+\sqrt{ac}| -n\log|t+\sqrt{bc}| + \log \frac{|\sqrt{b}(z_0\sqrt{a}+x_0\sqrt{c})|}{|\sqrt{a}(z_1\sqrt{b}+y_1\sqrt{c})|}\right| < K\sqrt{|ac|} |s+\sqrt{ac}|^{-m}.  \]
\end{lemma}

\section{Linear form in logarithms is non-zero}
Denote $ \ds \Lambda = \log{\frac{|Q|}{|P|}}= m\log|s+\sqrt{ac}| -n\log|t+\sqrt{bc}| + \log \frac{|\sqrt{b}(z_0^{(i)}\sqrt{a}+x_0^{(i)}\sqrt{c})|}{|\sqrt{a}(z_1^{(j)}\sqrt{b}+y_1^{(j)}\sqrt{c})|}$, so $\Lambda$ is the linear form in logarithms of algebraic numbers. This form is usually involved in solving the extension problems of Diophantine triples (for example, it was studied in \cite{Zrinka} and \cite{BFT}). The same linear form was also useful in solving some Thue equations \cite{BoZr}. It is usually shown that this form is not zero so that one can apply the famous Baker-W\" ustholz theorem \cite{bakerw} and subsequently, bound the coefficients $m$ and $n$. In rational integers, the proof that $\Lambda \neq 0$ is often trivial, but in quadratic fields it can cause considerable problems, as it happened in \cite{Zrinka} and \cite{BoZr}. With some mild conditions, we prove that $\Lambda \neq 0$, and this is valid for an arbitrary imaginary quadratic field $K$ and $a, b, c$ in its ring of integers $\mathcal{O}_K$.

\begin{lemma}\label{lema:fneq0}
If $\{a, b, c\}$ is an extensible Diophantine triple, $\ds \frac{|c|}{|a|} \not\in \Q$ and $\ds \frac{|c|}{|b|} \not\in \Q$, then $\Lambda \neq 0$.
\end{lemma}
\begin{proof}
The proof strategy is the same as in the proof of Lemma 5.2. in \cite{Zrinka}, but some details differ.
Let us prove that, if $v_m = w_n$, then $|P| \neq |Q|$. 
$P \neq Q$ is easy, because otherwise $P^{-1}=Q^{-1} $, and then $v_m=w_n$ would imply $\frac{c-a}{a}=\frac{c-b}{b}$ and $b=a$ (since $c\neq 0$).

%za sve $a, b, c\in\mathbb{Z}[i]$ modula većeg od $1$. %Jasno, da je $|a|=1$, bila bi prekršena pretpostavka leme da $\frac{|c|}{|a|} \not\in \Q$.

By definition, $P=A+B\alpha, Q=C+D\beta$, where $\alpha = \sqrt{\frac{c}{a}}$ and $\beta=\sqrt{\frac{c}{b}}$, $A, B, C, D\in K$.
From $v_m = w_n$, we get $\frac{A+B\alpha+A-B\alpha}{2} = \frac{C+D\beta+C-D\beta}{2}$, ($v_m = A, w_n = C$) so $A=C$. Furthermore,

$|P|^2 = (A+B\alpha)(\bar{A}+\bar{B}\bar{\alpha})=A\bar{A}+\bar{A}B\alpha+A\bar{B}\bar{\alpha}+|B|^2|\alpha|^2$ i.~e.\
\[ |P|^2 = p+u\alpha+\bar{u}\bar{\alpha}+q|\alpha|^2, \quad |Q|^2 = r+v\beta+\bar{v}\bar{\beta}+s|\beta|^2,\]
where $p, q, r, s\in\mathbb{Q}$ and $u, v\in K$. The idea is to show that the $1, \alpha, \bar{\alpha}, |\alpha|^2, \beta, \bar{\beta}$ and $|\beta|^2$ are linearly independent. We do this through three steps (claims \textbf{A, B} and \textbf{C}).

\textbf{Claim A: The numbers $\alpha = \sqrt{\frac{c}{a}}$, $\beta=\sqrt{\frac{c}{b}}$ and $\sqrt{\frac{a}{b}}$ are algebraic of degree $2$ over $K$.}

\noindent Lemma \ref{acnekva} implies that $\frac{c}{a}$, $\frac{c}{b}$ and $\frac{a}{b}$ are not squares in $K$.  \qed
%\vskip 1em

\textbf{Claim B: Basis for $K(\alpha, \bar{\alpha})$ over $K$ is $B_\alpha = \{1, \alpha, \bar{\alpha}, |\alpha|^2\}$. Analogously, basis for $K(\beta, \bar{\beta})$ is $B_\beta = \{1, \beta, \bar{\beta}, |\beta|^2\}$.}

If $\gamma\in K(\alpha, \bar{\alpha})$, then $\gamma=\sum q_{ij}\alpha^i\bar{\alpha}^j$, where $q_{ij}\in K$. However, since $\alpha^2=\frac{c}{a}$ and $\bar{\alpha}^2$ are in $K$ and $\alpha\bar{\alpha}=|\alpha|^2$, it follows that one can write $\gamma$ as $\gamma = q_0+q_1\alpha+q_2\bar{\alpha}+q_3|\alpha|^2$.

To prove that $B_\alpha$ is linearly independent set over $K$, we first show that $\{1, \alpha, \bar{\alpha}\}$ is linearly independent. Assume the contrary. Then $\bar{\alpha} = A+B\alpha$ for $A, B\in K$. This implies $\bar{\alpha}^2-A^2-B^2\alpha^2=2AB\alpha$. Hence, if $AB\neq 0$, then $\ds \alpha = \frac{\bar{\alpha}^2-A^2-B^2\alpha^2}{2AB} \in K$, which contradicts the claim \textbf{A}.  If $B=0$, then $\bar{\alpha}=A\in K$. This would imply that $\alpha = \sqrt{\frac{c}{a}}$ is in $K$, i.~e.~$\frac{c}{a} = \frac{x^2}{y^2}$ for some $x, y\in\mathcal{O}_K$, so $\frac{|c|}{|a|} =\frac{|x|^2}{|y|^2}\in \Q$, which contradicts the lemma hypothesis.

\noindent If $A=0$, then $\bar{\alpha}=B\alpha$, so again $\ds \frac{|c|}{|a|} = |\alpha|^2=\alpha\bar{\alpha}=B\alpha^2\in K \cap \R, \text{  i.~e.~again it follows that } \frac{|c|}{|a|}\in\Q$.

\noindent Therefore, $\{1, \alpha, \bar{\alpha}\}$ is linearly independent set over $K$.

For $B_\alpha$, it suffices to show that there are no $A, B, C\in K$ such that 
\begin{equation}\label{eq:al21}
|\alpha|^2=A+B\alpha+C\bar{\alpha}.
\end{equation}
We first prove $C\neq 0$. The contrary would imply  $|\alpha|^4=A^2+B^2\alpha^2+2AB\alpha$ and $2AB\alpha\in K$. Since $\alpha\not\in K$, it follows that $AB=0$. If $B=0$, then $|\alpha^2|=A\in\Q$, which contradicts the lemma hypothesis. If $A=0$, then $|\alpha|^2=B\alpha$, so $\bar{\alpha}=B\in K$,  which contradicts the claim $\textbf{A}$. Therefore, $C\neq 0$. %the lemma hypothesis. Therefore, $C\neq 0$.

By multiplying \eqref{eq:al21} by $\alpha$, we get $\alpha^2\bar{\alpha}=A\alpha+B\alpha^2+C|\alpha|^2$, which implies $C|\alpha|^2=-A\alpha-B\alpha^2+\alpha^2\bar{\alpha}$, i.~e.~$\ds |\alpha|^2 = -\frac{B}{C}\alpha^2-\frac{A}{C}\alpha+\frac{1}{C}\alpha^2\bar{\alpha}$.
%\end{equation}
Since $\{1, \alpha, \bar{\alpha}\}$ is linearly independent, the last obtained equality together with \eqref{eq:al21} implies that $A=-\frac{B}{C}\alpha^2, B=-\frac{A}{C}, C=\frac{1}{C}\alpha^2$. This implies $C^2=\alpha^2$, which contradicts the claim \textbf{A} ($\alpha^2$ is not a square in $K$). \qed
\vskip 0.2em
\textbf{Claim C: The set $B=\{1, \alpha, \bar{\alpha}, |\alpha|^2, \beta, \bar{\beta}, |\beta|^2\}$ is linearly independent over $K$.}

First we show that $\beta, \bar{\beta}$ and $|\beta|^2$ are not in $K(\alpha, \bar{\alpha})$. Let us assume that $\beta$ can be written as \begin{equation}\label{eq:beta} \beta=A+B\alpha+C\bar{\alpha}+D|\alpha|^2,\end{equation}
for some $A, B, C, D \in K$. Then
\[ \beta^2 = A^2+B^2\alpha^2+C^2\bar{\alpha}^2+D^2|\alpha|^4+2AB\alpha+2AC\bar{\alpha}+2AD|\alpha|^2+2BC|\alpha|^2+2BD\alpha^2\bar{\alpha}+2CD\bar{\alpha}^2\alpha. \]

\noindent By $\beta^2\in K$, it follows that the coefficients of algebraic numbers $\alpha, \bar{\alpha}$ and $|\alpha|^2$ are zero, i.~e.
\begin{align}
AB+CD\bar{\alpha}^2 &= 0, \label{eqC1}\\
AC+BD\alpha^2 &= 0\label{eqC2},\\
AD+BC &= 0.\label{eqC3}
\end{align}
By \eqref{eqC1} and \eqref{eqC3}, multiplying \eqref{eqC1} by $A$ or $C$, it follows that $A^2=C^2\bar{\alpha}^2$ or $B=D=0$.
Similarly, from \eqref{eqC2} and \eqref{eqC3}, it follows that $A^2=B^2\alpha^2$ or $C=D=0$,
while \eqref{eqC1} and \eqref{eqC2} imply that $A^2=D^2|\alpha|^4$ or $B=C=0$. There are four cases now.
\begin{itemize}
\item $B=C=D=0$. By \eqref{eq:beta}, it follows that $\beta\in K$, which contradicts the claim \textbf{A}.
\item $A=B=C=0$. Then $\beta = D|\alpha|^2$ and $\frac{|c|}{|b|} = |D|^2 \frac{|c|^2}{|a|^2}\in \Q$, which contradicts the lemma hypothesis.
\item $B\neq 0$, $C=D=0$. Hence $\beta=B\alpha$, i.~e.\ $\sqrt{\frac{a}{b}}=B\in K$, which contradicts the claim \textbf{A}.
\item $B\neq 0$ and at least one of $C$ and $D$ is non-zero. Then $A^2=C^2\bar{\alpha}^2=B^2\alpha^2=D^2|\alpha|^4$, so $\beta^2=4A^2$, which again contradicts \textbf{A}.
\end{itemize}
Therefore, $\beta$ cannot be written as a linear combination of elements in $B_\alpha$. The same holds for $\bar{\beta}$ and $|\beta|^2$ and is proven identically.

The set $L[\{1, \alpha, \bar{\alpha}, |\alpha|^2\}]$ (spanned by $B_\alpha$) is closed on inversion. Namely,
\[ \frac{1}{A+B\alpha+C\bar{\alpha }+D|\alpha|^2} =\frac{ (A+B\alpha)-(C \bar{\alpha}+D|\alpha |^2)}{K+L\alpha} = \frac{((A+B\alpha)-(C\bar{\alpha }+D|\alpha|^2))(K-L\alpha)}{K^2-L^2\alpha^2},\]
where $K= A^2+B^2\alpha^2-C^2\bar{\alpha}^2-D^2|\alpha|^4$, $L=2(AB-CD \bar{\alpha}^2)$.

%Pokažimo sada da baza ne može imati samo pet elemenata, odnosno biti oblika $\{1, \alpha, \bar{\alpha}, |\alpha|^2, \beta\}$ (ili $B_\alpha\cup\{\bar{\beta}\}$, $B_\alpha\cup\{ |\beta|^2 \}$, $B_\beta\cup\{\alpha\}$, $B_\beta\cup\{\bar{\alpha}\}$, $B_\beta\cup\{|\alpha|^2\}$).

Now we show that $\bar{\beta}$ cannot be written as linear combination of elements in $B_\alpha\cup\{\beta\}$. Analogously one shows the linear independence of sets $B_\alpha\cup\{\bar{\beta}, |\beta|^2\}$ and $B_\alpha\cup\{\beta, |\beta|^2\}$. Namely, that implies $\bar{\beta}= q_1+q_2\alpha+q_3\bar{\alpha}+q_4|\alpha|^2 + q_5\beta$ and $q_5\neq 0$. %, odnosno $\bar{\beta}=u+q_5\beta$, gdje je $u\in L[\{1, \alpha, \bar{\alpha}, |\alpha|^2\}]$.
Therefore
\begin{align*}
\bar{\beta}^2 =& q_1^2+q_2^2\alpha^2+q_3^2\bar{\alpha}^2+q_4|\alpha|^4+q_5^2 +2(q_1q_2+q_3q_4\bar{\alpha}^2)\alpha+2(q_1q_3+q_2q_4\alpha^2)\bar{\alpha}+2(q_1q_4+q_2q_3)|\alpha|^2+\\
&\, +2q_1q_5\beta+2(q_2q_5\alpha+q_3q_5\bar{\alpha})\beta+2q_4q_5|\alpha|^2\beta,
\end{align*}
so we see that $2q_5\beta(q_1+q_2\alpha+q_3\bar{\alpha}+q_4|\alpha|^2)\in L[\{1, \alpha, \bar{\alpha}, |\alpha|^2\}]$. By $q_5\neq 0$, it follows that $q_1+q_2\alpha+q_3\bar{\alpha}+q_4|\alpha|^2=0$, i.~e.~ $q_1=q_2=q_3=q_4=0$. However, that means $\bar{\beta}=q_5\beta$ for some $q_5\in K$, which implies $|\beta|^2=q_5\beta^2\in K\cap\R$, i.~e.~$|\beta|^2\in\Q$, contradicting the lemma hypothesis.

We get the contradiction in a similar way if we assume that $|\beta|^2$ can be written as linear combination of elements in $\{1, \alpha, \bar{\alpha}, |\alpha|^2, \beta, \bar{\beta}\}$. By $|\beta|^2=q_1+q_2\alpha+q_3\bar{\alpha}+q_4|\alpha|^2 + q_5\beta+q_6\bar{\beta}$, it follows that $2(q_1+q_2\alpha+q_3\bar{\alpha}+q_4|\alpha|^2+q_5q_6)(q_5\beta+q_6\bar{\beta}) \in L[\{1, \alpha, \bar{\alpha}, |\alpha|^2\}]$, so $q_5\beta+q_6\bar{\beta}=0$, which would again imply $|\beta|^2\in\Q$ or $q_1+q_2\alpha+q_3\bar{\alpha}+q_4|\alpha|^2+q_5q_6 = 0$. Since $B_\alpha$ is linearly independent, it follows that $q_2=q_3=q_4=0$ and $q_1+q_5q_6=0$. Hence $|\beta|^2=q_1+q_5\beta+q_6\bar{\beta}$, but this contradicts the linear independence of $B_\beta$.
\qed
\vskip 0.2em
Let us remind the reader that, prior to these three claims, we have shown that, from $|P|^2 = (A+B\alpha)(\bar{A}+\bar{B}\bar{\alpha})$ and a similar equality for $|Q|^2$, it follows that
\[ |P|^2 = p+u\alpha+\bar{u}\bar{\alpha}+q|\alpha|^2, \quad |Q|^2 = r+v\beta+\bar{v}\bar{\beta}+s|\beta|^2,\]

where $p, q, r, s\in\mathbb{Q}$, and $u, v\in K$. Since we want to prove that $|P|\neq |Q|$, it suffices to show that $|P|^2\neq|Q|^2$. If $|P|^2=|Q|^2$, this would imply $(p-r)+u\alpha+\bar{u}\bar{\alpha}+q|\alpha|^2-v\beta-\bar{v}\bar{\beta}-s|\beta|^2=0$, so the claim \textbf{C} implies that $p-r=u=q=v=s=0$, i.~e.~$P=A=C=Q$, which we have already proven to be impossible. Therefore $|P|\neq|Q|$, which implies that $\ds \Lambda =\log\frac{|P|}{|Q|}\neq 0$.
\end{proof}

The statement of this lemma depends on the system of equations chosen at the beginning. However, one easily sees that the analogous claim holds even if one begins with a different system (e.~g.~$az^2-cx^2=a-c, ay^2-bx^2=a-b$). Choosing which system to deal with usually depends on being able to find all the fundamental solutions for one of the equations. Regardless of which system is chosen, one can use this lemma.
%Ova lema može biti korisna i kod proučavanja proširenja jednoparametarskih familija Diofantovih trojki u Gaussovim cijelim brojevima. Npr.\ specijalna verzija te leme dokazana je i korištena u \cite{Zrinka}, a korištena je i u \cite{BFT}. Uz to, u oba ta članka korištena je tvrdnja slična Lemi \ref{lforma}. I u nekim drugačijim istraživanjima dokazivan je analogni rezultat, npr.~u \cite{BoZr}.

%\section{
\section{System of Pell-type equations for triples of the form $\{k-1, k+1, 16k^3-4k\}$ }
The set $\{k-1, k+1, 16k^3-4k\}$ is a Diophantine triple for every Gaussian integer $k$. Denote by $s=4k^2-2k-1, t=4k^2+2k-1$ (so $(k-1)(16k^3-4k)+1=s^2$ and $(k+1)(16k^3-4k)+1 = t^2$). Assume now that $d$ extends this Diophantine triple, i.~e., that $\{k-1,k+1,16k^3-4k,d\}$ is a Diophantine quadruple in $\Z[i]$. There exist $x,y,z\in \Z[i]$ such that
\[ (k-1)d+1=x^2,\ (k+1)d+1=y^2,\ (16k^3-4k)d+1=z^2. \]
By eliminating $d$, we obtain the system
\begin{align}
(k+1)x^2-(k-1)y^2&=2,\label{e1}\\
(16k^3-4k)x^2-(k-1)z^2&=16k^3-5k+1\label{e2}
\end{align}

Let  $|k|>3$. By Lemma 2.4 of \cite{Zrinka}, all the solutions of the equation \eqref{e1} are given by $x=\pm V_n$, where $(V_n)$ is a recurrent sequence defined by
\begin{equation}\label{eq:nizv4}
 V_0=1,\ V_1=2k-1,\ V_{n+2}=2kV_{n+1}-V_n,\text{ for all } n\geqslant 0. \end{equation}
All solutions of the equation \eqref{e2} are described in the following lemma, which follows from Lemma~\ref{lem:fund}. %(and can be proven completely analogously as the Lemma~2.2 in \cite{Zrinka}).
\begin{lemma}\label{l1}
Let $k\in\Z[i]\backslash\{0,1\}$. There are $j_0\in\N$, $x_1^{(j)}, z_1^{(j)}\in\Z[i]$, $j=1,\ldots,j_0$ such that
\begin{enumerate}
\item[a)] $(x_1^{(j)}, z_1^{(j)})$ is the solution of (\ref{e2}) for all $j=1,\ldots,j_0$,
\item[b)] these fundamental solutions are bounded as follows:
\begin{eqnarray*}
|x_1^{(j)}|^2 &\leqslant &\frac{|16k^3-5k+1||k-1|}{|4k^2-2k-1|-1}\\
|z_1^{(j)}|^2 &\leqslant &\frac{|16k^3-4k||16k^3-5k+1|}{|4k^2-2k-1|-1}+\frac{|16k^3-5k+1|}{|k-1|}
\end{eqnarray*}
for all $j=1,\ldots,j_0$,
\item[c)] for each solution $(x,z)$ of \eqref{e2} there are $j\in\{1,\ldots,j_0\}$ and $m\in\Z$ such that
\[ x\sqrt{16k^3-4k}+z\sqrt{k-1}=(x_1^{(j)}\sqrt{16k^3-4k}+z_1^{(j)}\sqrt{k-1})\cdot\left(4k^2-2k-1+\sqrt{(k-1)(16k^3-4k)}\right)^m. \]
\end{enumerate}
\end{lemma}
Hence, the solution $x$ of the equation \eqref{e2} is $x=\pm W_m^{(j)}$ for some $j\in\{1,\ldots,j_0\}$ and $m\in\N_0$, where the sequence $(W_m^{(j)})_m$ is recurrently defined by
\[ W_0^{(j)}=x_1^{(j)},\ W_1^{(j)}=x_1^{(j)}(4k^2-2k-1)+z_1^{(j)}(k-1),\ W_{m+2}^{(j)}=2(4k^2-2k-1)W_{m+1}^{(j)}-W_{m}^{(j)},\ m\geqslant 0. \]
For the time being, we omit the upper index $(j)$.

If $x$ is the solution of both \eqref{e1} and \eqref{e2}, then $x=V_n=W_m$. We are looking for the common elements of the sequences $(V_n)_n$ and $(W_m)_m$.

We apply the congruence method now. Observe the remainders that $(V_n)$ and $(W_m)$ leave when divided by $s=4k^2-2k-1$. The following lemma is easily proven by induction.
 \begin{lemma} For the sequences $(V_n)_n$ and $(W_m)_m$ it holds
 \begin{align*}
V_n \equiv 0, \pm 1, \pm(2k-1) & \pmod{4k^2-2k-1} \text{  and}\\
W_m \equiv \pm x_1, \pm z_1(k-1)& \pmod{4k^2-2k-1}
\end{align*}
for all indices $n$ and $m$.
\end{lemma}
 
By analysing these combinations we can conclude that, when $|k|>17$, all fundamental solutions $(x_1, z_1)$ which generate sequences $(W_m)_m$ that can intersect the sequence $(V_n)_n$, are given by the set
\[ (x_1,z_1)\in \{(\pm 1,\pm 1),(\pm k,\pm(4 k^2+ 2 k-1)),(\pm(2k-1),\pm(8k^2-1))\}. \]

E.~g.\ if $x_1\equiv 0 \pmod {4k^2-2k-1}$, then $x_1=0$ or $|x_1|\geqslant 4|k|^2-2|k|-1$. However, in the latter case, the bound given in Lemma \ref{l1} implies that \[ (4|k|^2-2|k|-1)^2 \leqslant \frac{|16k^3-5k+1||k-1|}{|4k^2-2k-1|-1}, \] i.~e.~$(4|k|^2-2|k|-1)^2(|4k^2-2k-1|-1) \leqslant |16k^3-5k+1||k-1|$, which implies $16|k|^4+16|k|^3+5|k|^2+4|k|+1 \geqslant 64 |k|^6 - 96 |k|^5 - 16 |k|^4 - 56 |k|^3 - 4 |k|^2 - 10 |k| - 2$, and is in turn equivalent to $-64 |k|^6 + 96 |k|^5 + 32 |k|^4 + 72 |k|^3 + 9 |k|^2 + 14 |k| + 3 \geqslant 0$, which is obviously impossible for large $|k|$ (one can determine that the largest zero of the left-hand side polynomial in $|k|$ is approximately $2.04414$). For $x_1\equiv \pm 1, \pm(2k-1)$, we similarly exclude all the possibilities except $x_1=\pm 1$ and $x_1=\pm(2k-1)$, which gives us the solutions $(\pm 1, \pm 1), (\pm (2k-1), \pm (8k^2-1))$.

If $z_1(k-1)\equiv 0, \pm1, \pm(2k-1) \pmod{4k^2-2k-1}$, then $z_1 = u(4k^2-2k-1)+r$ where $u$ is a Gaussian integer, while $r\in\{0, \pm 4k, \pm (4k+2)\}$, since $-4k-2$ is the multiplicative inverse of $k-1$ modulo $4k^2-2k-1$. The equation \eqref{e2} implies $k\mid 1-z_1^2$. On the other hand, $z_1\equiv -u+r \equiv -u, -u\pm 2\pmod{k}$, so $k \mid 1-u^2$ or $k\mid 1-(u\pm 2)^2$. If $|u| \leqslant 2$, then $|1-u^2|\leqslant 1+|u|^2\leqslant 5$ and $|1-(u\pm 2)^2| \leqslant |u|^2+4|u|+5\leqslant 17$, and the obtained divisibility cannot hold if $|k| > 17$, except when $u=\pm 1$. Here we get the solutions $(\pm k,\pm(4 k^2+ 2 k-1))$ for $u=\pm 1$ and $z_1=u(4k^2-2k-1)$. It is not possible that both $u=\pm 1$ and $r=\mp(4k+2)$ hold, because then from the equation \eqref{e2} it follows that $x_1^2 = \frac{(k-1)(4k^2-6k-3)^2+(16k^3-5k+1)}{16k^3-4k} \in\Z[i]$. This implies $2k^2-k\mid 8k+8$, which is impossible for $|k|>17$ (since $8k+8=0$ or $8|k|+8\geqslant 2|k|^2-|k|$).
If $|u|\geqslant \sqrt{5}$,  repeating the juxtaposition with upper bound from Lemma \ref{l1}, we see that this cannot hold for $|k| > 17$: $|z_1|\geqslant \sqrt{5}(4|k|^2-2|k|-1)-4|k|-2=4\sqrt{5}|k|^2-(4+2\sqrt{5})|k|-(2+\sqrt{5})$, so Lemma \ref{l1} implies \begin{align*} & 80 |k|^4 - 32 \sqrt{5} |k|^3 - 80 |k|^3 - 4 |k|^2 + 16 \sqrt{5} |k| + 36 |k| + 4 \sqrt{5} + 9 \leqslant \\
& \quad \leqslant \frac{(16|k|^3+4|k|)(16|k|^3+5|k|+1|}{4|k|^2-2|k|-2}+\frac{16|k|^3+5|k|+1|}{|k|-1}.
\end{align*} It follows that $-64 |k|^7 + (544 + 128 \sqrt{5}) |k|^6 + (-256 - 192 \sqrt{5}) |k|^5 + (-488 - 64 \sqrt{5}) |k|^4 + (332 + 144 \sqrt{5}) |k|^3 + (40 + 24 \sqrt{5}) |k|^2 + (-88 - 32 \sqrt{5}) |k| - 8 \sqrt{5} - 20 \leqslant 0$, which does not hold for $|k| > 12.019$.

 This proves the following lemma.

\begin{lemma}\label{lema:nizw} If $|k|>17$ and the system of equations \eqref{e1} and \eqref{e2} has a solution $x\not=\pm1$, then there are positive integers $m$ and $n$, and $1\leqslant j\leqslant 6$, such that
\[ V_n=\pm W_m^{(j)}, \]
where the sequences $(W_m^{(j)})$ are given with the following initial conditions
\begin{eqnarray*}
 W_0^{(1)}=1,& W_1^{(1)}=&4 k^2- k-2,\\
 W_0^{(2)}=1,& W_1^{(2)}=&4 k^2-3 k,\\
 W_0^{(3)}=k,& W_1^{(3)}=&8 k^3- 4 k^2- 4 k+1 ,\\
W_0^{(4)}=k,& W_1^{(4)}=&2k-1,\\
W_0^{(5)}=2k-1,& W_1^{(5)}=&16 k^3- 16 k^2 - k+2 ,\\
W_0^{(6)}=2k-1,& W_1^{(6)}=&k,
\end{eqnarray*}
and all the other elements are defined by
\[ W_{m+2}^{(j)}=2(4k^2-2k-1)W_{m+1}^{(j)}-W_{m}^{(j)},\ m\geqslant 0. \]
\end{lemma}
Observe that the sequences $(W_m^{(j)})_{j=1,\dotsc,6}$ intersect with $(V_n)_n$ at  $\{1\}$, $\{1\}$, $\{8 k^3- 4 k^2- 4 k+1\}$, $\{2k-1\}$, $\{2 k-1\}$, $\{2 k-1,8 k^3- 4 k^2- 4 k+1\}$, respectively. These intersections correspond to the extensions $d\in\{0, 4k, 64k^5 -48k^3+8k\}$. %(od kojih prva dva ne čine četvorku).

\section{Lower bound for the solutions}
With the aim of obtaining a lower bound for the solution $|x|$, we determine the remainders of the elements of sequences from Lemma \ref{lema:nizw} modulo $4k(k-1)$. In this section, unless otherwise specified, we assume that $|k|>17$. By calculating the first few elements of the sequences, we get
\begin{align*}
 (V_n)_{n\geqslant 0}&\equiv&&(1, 2k-1, 2k-1, 1, 1, 2k-1,\ldots) \hfill  &\pmod{4k(k-1)}, \\
 (W_m^{(1)})_{m\geqslant 0}&\equiv&&(1, 3k-2, -2k+3, 5k-4, -4k+5, 7k-6,\ldots)\hfill  &\pmod{4k(k-1)}, \\
 (W_m^{(2)})_{m\geqslant 0}&\equiv&&(1, k, 2k-1, -k+2, 4k-3, -3k+4,\ldots)\hfill  &\pmod{4k(k-1)},  \\
(W_m^{(3)})_{m\geqslant 0}&\equiv&&(k, 1, 3k-2, -2k+3, 5k-4, -4k+5,\ldots)\hfill & \pmod{4k(k-1)},  \\
(W_m^{(4)})_{m\geqslant 0}&\equiv&&(k, 2k-1, -k+2, 4k-3, -3k+4, 6k-5,\ldots)\hfill & \pmod{4k(k-1)},  \\
 (W_m^{(5)})_{m\geqslant 0}&\equiv&&(2k-1, -k+2, 4k-3, -3k+4, 6k-5,\ldots)\hfill & \pmod{4k(k-1)},  \\
 (W_m^{(6)})_{m\geqslant 0}&\equiv&&(2k-1, k, 1, 3k-2, -2k+3, 5k-4,\ldots)\hfill & \pmod{4k(k-1)}. \end{align*}
%pa se mo\v ze dobiti sli\v cna ograda kao u \cite{Zrinka}.

\begin{lemma}\label{lema:kongrvw}
Let $k$ be a Gaussian integer (of absolute value greater than $1$). For the sequence $(V_n)_n$ defined in \eqref{eq:nizv4}, it holds that $V_n \equiv 1 \pmod{4k(k-1)}$ if $n\equiv 0, 3 \pmod{4}$, while $V_n \equiv 2k-1 \pmod{4k(k-1)}$ for $n\equiv 1, 2\pmod{4}$.

For the sequence $(W_m^{(1)})_m$ defined in Lemma \ref{lema:nizw}, its elements $W_{2m}^{(1)} \equiv -2mk+2m+1 \pmod{4k(k-1)}$ and $W_{2m+1}^{(1)} \equiv (2m+3)k-2m-2 \pmod{4k(k-1)}$ for all $m \in \N_0$. This sequence $(W_m^{(1)})_m$ is increasing in absolute value, and the inequality $|W_m^{(1)}| \geqslant (8|k|^2-4|k|-3)^{m-1}$ holds for all $m\geqslant 0$. Similarly, the other sequences $(W_m^{(j)})_m$, for $j=2, \dotsc, 6$,  are increasing (in absolute value) after the index $m=1$ and $|W_m^{(j)}| \geqslant (8|k|^2-4|k|-3)^{m-1}$. The following congruences also hold
\begin{align*}
& W_{2m}^{(2)}\equiv 2mk-(2m-1),& &W_{2m+1}^{(2)}\equiv -(2m-1)k+2m \quad\quad &\pmod{(4k(k-1)}\\
& W_{2m}^{(3)}\equiv (2m+1)k-2m,& &W_{2m+1}^{(3)}\equiv -2mk+2m+1 \quad\quad &\pmod{(4k(k-1)}\\
& W_{2m}^{(4)}\equiv (1-2m)k+2m,& &W_{2m+1}^{(4)}\equiv 2mk-(2m+1) \quad\quad &\pmod{(4k(k-1)}\\
& W_{2m}^{(5)}\equiv (2m+2)k-(2m+1),& &W_{2m+1}^{(5)}\equiv (-2m-1)k+2m \quad\quad &\pmod{(4k(k-1)}\\
& W_{2m}^{(6)}\equiv (2-2m)k+(2m-1),& &W_{2m+1}^{(6)}\equiv (2m+1)k-2m \quad\quad &\pmod{(4k(k-1)}
\end{align*}
\end{lemma}
\begin{proof}
All the claims are proven inductively. We first prove that the sequence $(|W_m^{(1)}|)_m$ is increasing.

For $m=1$, the inequality $|W_1| \geqslant |W_0|$ holds since $|4k^2-k-2| \geqslant 4|k|^2-|k|-2 \geqslant 1 = |W_0|$ for $|k| \geqslant 1$.

\noindent From $W_{m+1}=2(4k^2-2k-1)W_{m}-W_{m-1}$, it follows that $|W_{m+1}| \geqslant (8|k|^2-4|k|-2)|W_m|-|W_{m-1}|  \geqslant (8|k|^2-4|k|-3)|W_m| + |W_m|-|W_{m-1}| \geqslant (8|k|^2-4|k|-3)|W_m|$.

This directly shows not only that the sequence of  absolute values is increasing, but also that $|W_m| \geqslant (8|k|^2-4|k|-3)^{m-1}$ for all $m\geqslant 0$. Likewise, this holds for $W_m^{(i)}$ for all $i = 1, \dotsc, 6$.
%\vspace{0.4cm}

Inductively one shows that $W_{2m}^{(1)} \equiv -2mk+2m+1 \pmod{4k(k-1)}$ and \\ $W_{2m+1}^{(1)} \equiv (2m+3)k-2m-2 \pmod{4k(k-1)}$. The inductive basis, for $m=0$, was computed before the lemma statement. If we assume that $W_{2m}^{(1)} \equiv -2mk+2m+1 \pmod{4k(k-1)}$ and $W_{2m+1}^{(1)} \equiv (2m+3)k-2m-2 \pmod{4k(k-1)}$, then
\begin{align*}
W_{2m+2}^{(1)}&= 2(4k^2-2k-1)W_{2m+1}-W_{2m}\\
&\equiv 2(2k-1)\cdot ((2m+3)k-2m-2) - (-2mk+2m+1)\\
&= (2m+3)(4k^2-2k)-(2m+2)(4k-2)+2mk-2m-1\\
&\equiv (2m+3)\cdot 2k - 8mk+4m-8k+4+2mk-2m-1\\
%&= -2mk-2k+2m+3\\
&= -2(m+1)k+2m+3 \pmod{4k(k-1)}.
\end{align*} Similarly, by using this claim, one gets that
\begin{align*}
W_{2m+3}^{(1)} &= 2(4k^2-2k-1)W_{2m+2}-W_{2m+1}\\
&\equiv 2(2k-1)(-2(m+1)k+2m+3)-((2m+3)k-2m-2)\\
&= -2(m+1)(4k^2-2k)+2(2k-1)(2m+3)-(2m+3)k+2m+2\\
&\equiv -2(m+1)\cdot 2k + 8mk+12k-4m-6-2mk-3k+2m+2\\
%&= 2mk+5k-2m-4\\
&= (2m+5)k-2m-4 \pmod{4k(k-1)}.
\end{align*}
This shows the congruence claims for $(W_m^{(1)})_m$, while the claims stated for the other sequences are proven completely analogously.
\end{proof}

Now we observe the even indices case: $W_{2m}^{(1)} = V_{2n} \Rightarrow -2mk+2m+1 \equiv 1 \pmod{4k(k-1)}$ so $4k^2-4k$ divides $2mk-2m$, i.~e.~$2k(k-1)|m(k-1)$ and, most importantly, $2k \mid m$.
Analogously, $W_{2m} = V_{2n+1} \equiv 2k-1 \pmod{4k(k-1)}$ implies that $2k \mid m+1$. Any of these conclusions, $2k \mid m$ and $2k \mid m+1$, implies that
\begin{equation}
m \geqslant 2|k|-1,
\label{eq:mk}\end{equation} unless $m$ is such that the corresponding multiple of $2k$ is actually $0$. Therefore, if $m\neq 0, -1$, then \[ |x| \geqslant (8|k|^2-4|k|-3)^{4|k|-3}.\]

For odd indices in the sequence $(W^{(1)}_m)$, $(2m+3)k-2m-2 \equiv 1 \pmod{4k(k-1)}$ implies that $4k(k-1)\mid (2m+3)(k-1)$ and $4k \mid 2m+3$, which is obviously impossible. Analogously, $(2m+3)k-2m-2 \equiv 1 \pmod{4k(k-1)}$ implies $4k\mid 2m+1$, a contradiction.

%\vspace{.4cm}
In the same way, one gets similar conclusions for the remaining sequences $(W_m^{(i)})_m$ ($i=2,\dotsc, 6$): for one case (even/odd index) there is a contradiction, while the other case implies that $2k\mid m$ or $2k\mid m\pm 1$. In any case, $m \geqslant 2|k|-1$ if $m\not\in\{-1, 0, 1\}$ and the same lower bound holds. We have proven the following result.

\begin{proposition}\label{prop:x} If $(x, y, z)$ is the solution of the system
\begin{align*}
(k+1)x^2-(k-1)y^2&=2,\tag{\ref{e1}}\\
(16k^3-4k)x^2-(k-1)z^2&=16k^3-5k+1 \tag{\ref{e2}},
\end{align*}
for $|k|>17$ and $x\not\in\{1,  k,  2k-1, 8k^3-4k^2-4k+1\}$, then $|x| \geqslant (8|k|^2-4|k|-3)^{4|k|-3}$.
 \end{proposition}
 
We note here that the exceptions $x= 1$, $x=k$, $x=2k-1$ and $x=8k^3-4k^2-4k+1$ correspond to the indices $m=0$ and $m=1$, i.~e.~when $2k \mid m$ does not imply that $m\geqslant 2|k|$.

\section{The problem of applying Jadrijević--Ziegler theorem}
%Mogu se dobiti slične nejednakosti kao u dosadašnjim radovima u $\Z[i]$.\\
There are two essentially different systems we can attempt to solve in this problem. One is given in Proposition \ref{prop:x} and Jadrijević--Ziegler theorem \cite{Borka} cannot be applied here because its conditions are not satisfied. The second system has coefficient $16k^3-4k$ on left-hand side of both of the equations -- we will show that, while the conditions are satisfied, this theorem cannot give us a useful result.

First, we focus on the system already given.
\begin{lemma}\label{thete12f}
If $(x, y, z)$ is a solution of the system of equations \eqref{e1} and \eqref{e2}, and
$\displaystyle \theta_1^{(1)} = \pm \sqrt{\frac{k+1}{k-1}}$, $\theta_1^{(2)} = -\theta_1^{(1)}$, 
$\displaystyle \theta_2^{(1)} = \pm \sqrt{\frac{4k^2-1}{4k(k-1)}}$, $\theta_2^{(2)} = -\theta_2^{(1)}$, where signs are chosen in such a way that
\[ \left| \theta_1^{(1)} - \frac{y}{x} \right| \leqslant \left| \theta_1^{(2)} - \frac{y}{x} \right| \quad \text{ and }\quad \left| \theta_2^{(1)} - \frac{z}{4kx} \right|\leqslant \left| \theta_2^{(2)} - \frac{z}{4kx} \right|, \]
then \begin{center}
\begin{align*}
\left| \theta_1^{(1)} - \frac{4ky}{4kx} \right| \leqslant & \frac{2}{\sqrt{|k^2-1|}}\cdot \frac{1}{|x|^2}, \quad \text{and}\\
\left| \theta_2^{(1)} - \frac{z}{4kx} \right| \leqslant & \frac{|16k^3-5k+1|}{8\sqrt{|4k^6-4k^5-k^4+k^3|}} \cdot \frac{1}{|x|^2}.
\end{align*}
\end{center}
\end{lemma}
\begin{proof}
The first inequality, $\displaystyle \left| \theta_1^{(1)} - \frac{4ky}{4kx} \right| \leqslant \frac{2}{\sqrt{|k^2-1|}}\cdot \frac{1}{|x|^2}$, was already obtained in \cite{Zrinka}. In the same manner,

\begin{align}
\left| \theta_2^{(1)} - \frac{z}{4kx} \right| &\,\,\,= \left| (\theta_2^{(1)})^2 - \frac{z^2}{16k^2x^2} \right|\cdot \left| \theta_2^{(1)} + \frac{z}{4kx} \right|^{-1} \nonumber\\
&\,\,\,= \left| \frac{1}{16k^2}\left(\frac{16k^3-4k}{k-1}-\frac{z^2}{x^2}\right)\right| \cdot \left| \theta_2^{(2)}-\frac{z}{4kx}\right|^{-1}\nonumber\\
&\overset{\eqref{e2}}{=} \frac{|16k^3-5k+1|}{|16k^3-16k^2|}\cdot \frac{1}{|x|^2} \cdot \left| \theta_2^{(2)}-\frac{z}{4kx}\right|^{-1}\label{theta2}
\end{align}

Furthermore, because of the way the signs were chosen, \[  \left| \theta_2^{(2)}-\frac{z}{4kx}\right| \geqslant \frac 12 \left(\left| \theta_2^{(1)}-\frac{z}{4kx}\right|+\left| \theta_2^{(2)}-\frac{z}{4kx}\right|\right) \geqslant \frac 12 |\theta_2^{(1)}-\theta_2^{(2)}| = \left| \sqrt{\frac{4k^2-1}{4k^2-4k}} \right| \]
Plugging in \eqref{theta2}, we get
\[ \left| \theta_2^{(1)} - \frac{z}{4kx} \right| \leqslant \frac{|16k^3-5k+1|}{|16k^3-16k^2|} \cdot \left| \sqrt{\frac{4k(k-1)}{4k^2-1}} \right| \cdot\frac{1}{|x|^2} = \frac{|16k^3-5k+1|}{8\sqrt{|4k^6-4k^5-k^4+k^3|}} \cdot \frac{1}{|x|^2}.\]
\end{proof}
Now we want to apply the following theorem \cite{Borka}. %Kako je ovaj teorem verzija poznatog Bennettovog teorema iz \cite{Bennett}, onda se u literaturi ponekad kaže ``primjenom Bennettovog teorema'', iako se zapravo odnosi na ovaj teorem. Mi ćemo izbjegavati takv
%\newpage
\begin{comment}
Neka je $\theta_i = \sqrt{1+\frac{a_i}{T}}, i = 1,\dotsc, m$, gdje su $a_i$ u parovima različiti algebarski cijeli brojevi u imaginarnom kvadratnom polju $K$ i $a_0=0$. Nadalje, neka je i $T$ algebarski cijeli u $K$, takav da je $|T| > M=\max\{|a_1|,\dotsc, |a_m|\}$, \[ L=\frac{27}{16|a_1|^2|a_2|^2|a_1-a_2|^2}(|T|-M)^2 > 1. \]
\end{comment}
\begin{theorem}[{\cite[Theorem 7.1]{Borka}}]
Let $\theta_i = \sqrt{1+\frac{a_i}{T}}, i = 1,2$ where $a_1 \neq a_2$ and $T$ are in the ring of integers of an imaginary quadratic field $K$. Let $|T| > M=\max\{|a_1|,|a_2|\}$, \[ L=\frac{27}{16|a_1|^2|a_2|^2|a_1-a_2|^2}(|T|-M)^2 > 1. \]

Then \[ \max\left\{\left|\theta_1-\frac{p_1}{q}\right|, \left|\theta_2-\frac{p_2}{q}\right|\right\} > c|q|^{-\lambda},\]
holds for all algebraic integers $p_1, p_2, q\in K$, where $\ds \lambda = 1+\frac{\log P}{\log L}$, $c^{-1} = 4pP(\max\{1, 2l\})^{\lambda-1}$,

\[ l = \frac{27}{64}\frac{|T|}{|T|-M}, p=\sqrt{\frac{2|T|+3M}{2|T|-2M}},
P=16\frac{|a_1|^2|a_2|^2|a_1-a_2|^2}{\min\{|a_1|, |a_2|, |a_1-a_2|\}^3} (2|T|+3M).\]
\end{theorem}
Writing $\displaystyle \theta_1 =\sqrt{\frac{k+1}{k-1}} = \sqrt{1+\frac{2}{k-1}}, \quad \theta_2 = \sqrt{\frac{4k^2-1}{4k(k-1)}} = \sqrt{1+\frac{4k-1}{4k^2-4k}}$, we see that, to get the same denominators, we need to write $\theta_1$ as $\theta_1 = \sqrt{1+\dfrac{8k}{4k^2-4k}}$. Hence
$a_1 = 8k, a_2 = 4k-1, T = 4k^2-4k$. Then $M=\max\{|a_1|, |a_2|\} = 8|k|$. The inequality $|k-1| \geqslant |k|-1 > 2$ holds for $|k| > 3$, so $|T| = 4|k||k-1| > 8|k|=M$.

Unfortunately, $\displaystyle L = \frac{27}{16\cdot (8|k|)^2 \cdot |4k-1|^2 \cdot |4k+1|^2} (4|k^2-k|-8|k|)^2$ is, for large $k$, less than $1$ (since the degree of $|k|$ is $6$ in the denominator and $4$ in the numerator), while the condition of the theorem is $L>1$. Therefore, we cannot directly apply this theorem.

On the other hand, we can attempt to solve the following system
\begin{align}\label{e101}
(16k^3-4k)y^2-(k+1)z^2=16k^3-5k-1,\\ 
(16k^3-4k)x^2-(k-1)z^2=16k^3-5k+1\tag{\ref{e2}}
\end{align}
%$\{A,B,C\}\in\{k-1,k+1,16k^3-4k\}$
and define $\vartheta_1,\vartheta_2$ as
\[ \vartheta_1^2=1+\frac{1}{(k-1)(16k^3-4k)},\ \vartheta_2^2=1+\frac{1}{(k+1)(16k^3-4k)},\]
where the signs of $\vartheta_1$ and $\vartheta_2$  are chosen in the same manner as in Lemma \ref{thete12f}. In that case, by the notation of Jadrijević-Ziegler theorem, \[ a_1=k+1,\ a_2=k-1,\ T=(k^2-1)(16k^3-4k).\]
Remark 7.2 from \cite{Borka} shows that the condition $L > 1$ is fulfilled whenever
$|T| > (4M)^{3}$. Here, this inequality $ |(k^2-1)(16k^3-4k)|>|4(k+1)|^3$ holds for $k\geqslant 3.21 $.
%Naime, lijeva strana je $|16k^5-20k^3+4k| \geqslant 16|k|^5-20|k|^3-4|k|$, a desna je\\ $64|k^3+3k^2+3k+1| \leqslant 4(16|k|^3+48|k|^2+48|k|+16)$, pa je dovoljno dokazati \\ $4|k|^5-5|k|^3-|k| > 16|k|^3+48|k|^2+48|k|+16$ tj.~ $4|k|^5-21|k|^3-48|k|^2-49|k|-16 > 0$, što očito vrijedi za dovoljno velik $|k|$, a najveća realna nultočka ovog polinoma ($4x^5-21x^3-48x^2-49x-16$) približno iznosi $3.20929$.

Now we need to show that $\vartheta_1$ and $\vartheta_2$ can be approximated by the quotient of solutions (up to multiplication by some element of $\Q[i]$). More precisely, we will bound
$\ds \left|\vartheta_1-\frac{sx}{(k-1)z}\right|$  in the following lemma, where $s^2=(4k^2-2k-1)^2$, and a similar expression for $\vartheta_2$.

\begin{lemma}For $|k| \geqslant 5$,
\[ \max\left\{ \left|\vartheta_1^{(1)}-\frac{s(k+1)x}{(k-1)(k+1)z}\right|,\left|\vartheta_2^{(1)}-\frac{t(k-1)y}{(k-1)(k+1)z}\right| \right\}< 40|k|^2|z|^{-2},\]
where $t=4k^2+2k-1$.
\end{lemma}

\begin{proof}
\begin{align*}  &\left|\vartheta_1^{(1)}-\frac{sx}{(k-1)z}\right| = \\ &\quad = \left|(\vartheta_1^{(1)})^2-\frac{s^2x^2}{(k-1)^2z^2}\right|\cdot \left| \vartheta_1^{(1)} +\frac{sx}{(k-1)z} \right|^{-1}\\
&\quad=\left|1+\frac{1}{(k-1)(16k^3-4k)}-\frac{s^2x^2}{(k-1)^2z^2}\right|\cdot \left| \vartheta_1^{(2)} -\frac{sx}{(k-1)z} \right|^{-1}\\
&\quad= \left|\frac{(k-1)^2(16k^3-4k)z^2+(k-1)z^2-((k-1)(16k^3-4k)+1)(16k^3-4k)x^2}{(k-1)^2(16k^3-4k)z^2}\right|\cdot\\& \quad \quad \cdot \left| \vartheta_1^{(2)} -\frac{sx}{(k-1)z} \right|^{-1}\\
&\quad= \left|\frac{(k-1)(16k^3-4k)((k-1)z^2-(16k^3-4k)x^2)+(k-1)z^2-(16k^3-4k)x^2}{(k-1)^2(16k^3-4k)z^2}\right|\cdot\\& \quad \quad \cdot \left| \vartheta_1^{(2)} -\frac{sx}{(k-1)z} \right|^{-1}\\
&\quad= \left|\frac{((k-1)(16k^3-4k)+1)((k-1)z^2-(16k^3-4k)x^2)}{(k-1)^2-(16k^3-4k)z^2}\right|\cdot \left| \vartheta_1^{(2)} -\frac{sx}{(k-1)z} \right|^{-1}\\
&\,\, \overset{\eqref{e2}}{=}\left|\frac{s^2((k-1)-(16k^3-4k))}{(k-1)^2(16k^3-4k)}\right|\cdot \left| \vartheta_1^{(2)} -\frac{sx}{(k-1)z} \right|^{-1}\cdot|z|^{-2}.
\end{align*}
Since $t^2=(k+1)(16k^3-4k)+1=(4k^2+2k-1)^2$, in the same way it follows that
\[ \left|\vartheta_2^{(1)}-\frac{ty}{(k+1)z}\right|=\left|\frac{t^2((k+1)-(16k^3-4k))}{(k+1)^2(16k^3-4k)}\right|\cdot \left| \vartheta_2^{(2)} -\frac{ty}{(k+1)z} \right|^{-1}\cdot|z|^{-2}. \]
%\newpage
Analogously as before,
\[  \left| \vartheta_1^{(2)} -\frac{sx}{(k-1)z} \right| \geqslant \frac 12 |\vartheta_1^{(1)}-\vartheta_1^{(2)}| =\left|\sqrt{1+\frac{1}{(k-1)(16k^3-4k)}} \right|=\left|\frac{4k^2-2k-1}{\sqrt{(k-1)(16k^3-4k)}}\right|,\] 
\[  \left| \vartheta_2^{(2)} -\frac{ty}{(k+1)z} \right| \geqslant \frac 12 |\vartheta_2^{(1)}-\vartheta_2^{(2)}| =\left|\sqrt{1+\frac{1}{(k+1)(16k^3-4k)}} \right|=\left|\frac{4k^2+2k-1}{\sqrt{(k+1)(16k^3-4k)}}\right|.\]
Now
% Naime,
\begin{align*}
& \left|\frac{s^2((k-1)-(16k^3-4k))}{(k-1)^2(16k^3-4k)}\right|\cdot \left| \vartheta_1^{(2)} -\frac{sx}{(k-1)z} \right|^{-1} \leqslant \\
&\quad \left|\frac{s^2((k-1)-(16k^3-4k))}{(k-1)^2(16k^3-4k)}\right| \cdot \left|\frac{4k^2-2k-1}{\sqrt{(k-1)(16k^3-4k)}}\right|^{-1}\\
&\quad = \left|\frac{(4k^2-2k-1)^2(16k^3-5k+1)}{(k-1)^2(16k^3-4k)}\right|\cdot\left|\frac{\sqrt{(k-1)(16k^3-4k)}}{4k^2-2k-1}\right|\\
&\quad = \left|\frac{(4k^2-2k-1)(16k^3-5k+1)}{(k-1)^2(16k^3-4k)}\right|\cdot |\sqrt{(k-1)(16k^3-4k)}| \\
%&\quad = \left|\frac{64 k^5 - 32 k^4 - 36 k^3 + 14 k^2 + 3 k - 1}{16 k^5 - 32 k^4 + 12 k^3 + 8 k^2 - 4 k}\right|\cdot|\sqrt{16 k^4 - 16 k^3 - 4 k^2 + 4 k}|
&\quad = \left|\frac{64 k^5 - 32 k^4 - 36 k^3 + 14 k^2 + 3 k - 1}{16 k^5 - 32 k^4 + 12 k^3 + 8 k^2 - 4 k}\right|\cdot\sqrt{|k-1|\cdot|16k^3-4k|}\\
&\quad \leqslant \frac{64|k|^5+32|k|^4+36|k|^3+14|k|^2+3|k|+1}{16|k|^5-32|k|^4-12|k|^3-8|k|^2-4|k|} \sqrt{16|k|^4+16|k|^3+4|k|^2+|k|}\\
&\quad \leqslant \frac{64|k|^5+32|k|^4+36|k|^3+14|k|^2+3|k|+1}{16|k|^5-32|k|^4-12|k|^3-8|k|^2-4|k|}\cdot (4|k|^2+2|k|+1)
\end{align*} 
The last used inequality is easily proven by squaring it. It suffices to show that $(64|k|^5+32|k|^4+36|k|^3+14|k|^2+3|k|+1)(4|k|^2+2|k|+1)\leqslant 40|k|^2(16|k|^5-32|k|^4-12|k|^3-8|k|^2-4|k|)$, which is equivalent to $384 |k|^7 - 1536 |k|^6 - 752 |k|^5 - 480 |k|^4 - 236 |k|^3 - 24 |k|^2 - 5 |k| - 1 \geqslant 0$. Since $384 |k|^7 \geqslant 1920 |k|^6$ for $|k| \geqslant 5$, so it  suffices to show that $384|k|^6-752 |k|^5 - 480 |k|^4 - 236 |k|^3 - 24 |k|^2 - 5 |k| - 1\geqslant 0$. By repeating this argument, we get the proof of the desired inequality.
%Uvedimo polinom $f(x)=384 x^7 - 1536 x^6 - 752 x^5 - 480 x^4 - 236 x^3 - 24 x^2 - 5 x - 1$. Buduci da je $f(5)=3319874>0$, a $f'(x)=2688 x^6 - 9216 x^5 - 3760 x^4 - 1920 x^3 - 708 x^2 - 48 x - 5$
\end{proof}

%a prema varijanti Bennetta iz \cite{JZ} je 
%$$ \max\left\{ \left|\vartheta_1^{(1)}-\frac{sBx}{ABz}\right|,\left|\vartheta_2^{(1)}-\frac{tAy}{ABz}\right| \right\}>c(|k|)|ABz|^{-\lambda},$$

%\textcolor{black}{Ovdje je $\displaystyle c_2 = \frac{3\Gamma(3/2)}{4\sqrt{\pi}\Gamma(4)} = \frac{1}{16}$, pa je}
We now show that $\ds 2l = 2\cdot\frac{27}{64}\frac{|T|}{|T|-M} < 1$. This is equivalent to $27|T| <32|T|-32M$, i.~e. $ 32M < 5|T|$.%~\begin{equation} \label{eq:MT}
%\end{equation}
Since $M$ is the larger among the numbers $|k-1|$ and $|k+1|$, and both of them are less or equal to $|k|+1$ (by the triangle inequality), it follows that $M\leqslant |k|+1$. Therefore, we can show that $32(|k|+1) < 5|16k^5-20k^3+4|$, which holds for $|k| \geqslant 1.33$. Namely, $5|16k^5-20k^3+4| \geqslant 80|k|^5-100|k|^3-20$, so it suffices to show that $80|k|^5-100|k|^3-32|k|-52 > 0$, which holds for $k$ with large absolute value.%, a numerički je utvrđena najveća realna nultočka polinoma $80x^5-100x^3-32x-52$.

Now $ c = \frac{1}{4pP}, L = \frac{27(|T|-M)^2}{64|k^2-1|^2}, p = \sqrt{1+\frac{5M}{2|T|-2M}}, P = 8(2|T|+3M)|k^2-1|^2$, $q=(k-1)(k+1)z$.

\noindent If we try to apply the Jadrijević-Ziegler theorem, then
\[ \lambda = 1+ \frac{\log 8 + \log(2|T|+3M)+2\log|k^2-1|}{\log 27 + 2\log{(|T|-M)}-\log{64|k^2-1|^2}}, \]
and let \[ \Max = \max\left\{ \left|\vartheta_1^{(1)}-\frac{s(k+1)x}{(k-1)(k+1)z}\right|,\left|\vartheta_2^{(1)}-\frac{t(k-1)y}{(k-1)(k+1)z}\right| \right\}.  \]
Then \[ \Max > \frac{1}{4pP} |q|^{-\lambda} = \frac{\sqrt{2|T|-2M}}{32\sqrt{2|T|+3M}(2|T|+3M)|k^2-1|^2} |q|^{-\lambda}. \]

Since $M < \frac{5}{32} |T|$, it follows that $2|T|-2M > \frac{27}{16}|T|$, and $2|T|+3M < \frac{79}{32} |T| < \frac 52 |T|$. Hence
\[ \Max > \frac{\sqrt{\frac{27}{16}|T|}}{32\cdot\frac 52 \sqrt{\frac 52} |T|^{\frac 32}|k^2-1|^2} |q|^{-\lambda}= \frac{3\sqrt{3}}{160\sqrt{10}}\frac{ |T|^{-1}|q|^{-\lambda}}{|k^2-1|^2} = C |k^2-1|^{-\lambda-3} |16k^3-4k|^{-1} |z|^{-\lambda}, \] where $C = \frac{3\sqrt{3}}{160\sqrt{10}}$.
Now we can conclude that $ C |k^2-1|^{-\lambda-3} |16k^3-4k|^{-1} |z|^{-\lambda} < 40|k|^2 |z|^{-2}$, i.~e.
\[ |z|^{2-\lambda} < \frac{40}{C} |k|^2 |k^2-1|^{\lambda+3} |16k^3-4k| = \frac{6400\sqrt{10}}{3\sqrt{3}} |k|^2 |k^2-1|^{\lambda+3} |16k^3-4k|. \]

This inequality can be used to bound the magnitude of solution $|z|$ when $\lambda < 2$, because the left-hand side is then a positive power of $|z|$. 
The proof for lower bound on $|x|$ is easily modified for $|z|$. It is not hard to see that $|z| \geqslant |x|$, so we could use the same lower bound. Since this lower bound is exponential in $|k|$, if $\lambda$ were less than $2$, then we would get a polynomial upper bound for $|z|$ and juxtaposition of these two bounds would give us the upper bound for $|k|$. Unfortunately, $\lambda > 2$ here. Namely, this claim is equivalent to $P > L$ and $ 8(2|T|+3M)|k^2-1|^2 > \frac{27(|T|-M)^2}{64|k^2-1|^2}$, and $512(2|(k^2-1)(16k^3-4k)|+3M)|k^2-1|^4 > 27(|(k^2-1)(16k^3-4k)|-M)^2$. Since $M=\max\{|k-1|, |k+1|\}$ is linear in $k$, we can already see that the degree of $k$  is greater in the left-hand side ($13>10$). More precisely, left-hand side is
$512(2|(k^2-1)(16k^3-4k)|+3M)|k^2-1|^4 \geqslant  512(32|k|^5-40|k|^3-11|k|-3)(|k|^2-1)^4$, while $27(|(k^2-1)(16k^3-4k)|-M)^2 < 27(16|k|^5+20|k|^3+4|k|)^2$. It suffices to check that
\begin{align*}
& 16384 |k|^{13} - 86016 |k|^{11} - 6912 |k|^{10} + 174592 |k|^9 - 18816 |k|^8 - 165888 |k|^7 - 8112 |k|^6 +\\
& \quad + 64512 |k|^5 -13536 |k|^4 + 2048 |k|^3 + 5712 |k|^2 - 5632 |k| - 1536 > 0,
\end{align*}

\noindent which holds for $|k| \geqslant 1.82$. %(očito vrijedi za dovoljno velike $|k|$, a numerički je utvrđena najveća nultočka polinoma s lijeve strane u $k$ i iznosi približno $1.81696$).

To conclude, the gap between $16k^3-4k$ and $k+1$ is not large enough for exponent $\lambda$ to be less than $2$, and this makes it unlikely to use the usual approach by Diophantine approximation.

We note here that the similar problem of extending $D(4)$-triple $\{k'-2, k'+2, 4(k')^3-4k'\}$ in rational integers was studied in \cite{ljubica}. For even $k'=2k$, dividing by $2$, we get $D(1)$-triples having the same form as the triples studied in this paper. In \cite{ljubica}, problem was solved using a similar method we tried to apply here. An amelioration of the analogous theorem in $\Z$ was proven there for a specific situation (where numerators under the square root in $\theta_i$ equal exactly $k-2$ and $k+2$, while the denominator is divisible by $k^2-4$).

\section{Application of linear forms in logarithms to the family $\{k-1, k+1, 16k^3-4k\}$}
We continue dealing with the extensibility problem of Diophantine triples $\{k-1, k+1, 16k^3-4k\}$. Sequence $(V_n)_n$ is defined as in \eqref{eq:nizv4}, while $(W_m^{(i)})_m$ is defined in Lemma \ref{lema:nizw}. Let us remind ourselves that for this family, $a=k-1, b=k+1, c=16k^3-4k$ and $r=k, s= 4k^2-2k-1, t=4k^2+2k+1$.%, a $(W_m^{(i)})$ kao u Lemi \ref{lema:nizw}.
%We want to use the following well known theorem from \cite{bakerw}.
%Nakon što je pomoću teorema Jadrijević--Ziegler \cite{Borka} dokazan
%Uz pomoć ovog teorema dobivena je gornja granica na indeks $m$, što je onda iskorišteno da bi se opisala sva proširenja Diofantovih trojki iz jednoparametarskih familija.
\begin{lemma}\label{lema:vmn} If $V_n=\pm W_m^{(j)}$ for some $j, m, n \in \N_0$ and $|k|>2.5$, then $m\leqslant n \leqslant 3m+2$.
\end{lemma}
\begin{proof}
Reccurence relations and Lemma \ref{lema:kongrvw} inductively imply the following inequalities % \ref{lema:kongrvw}
\begin{align*}
(2|k|-1)^n &\leqslant |V_n| \leqslant (2|k|+1)^n \\
(8|k|^2-4|k|-3)^{m-1} &\leqslant |W_m^{(j)}| \leqslant (8|k|^2+4|k|+3)^{m+1}.
\end{align*}
If $V_n=W_m$, then $(2|k|+1)^n \geqslant (8|k|^2-4|k|-3)^{m-1}$, so $n\geqslant m$. If we assume the contrary, $n\leqslant m-1$, then $8|k|^2-4|k|-3 \leqslant 2|k|+1$, which creates a contradiction when $|k| > 2.5$.

We now assume $n\geqslant 3m+3$. From $V_n=W_m$ it follows that $(8|k|^2+4|k|+3)^{m+1} \geqslant (2|k|-1)^n\geqslant (2|k|-1)^{3m+3}$, so $8|k|^2+4|k|+3 > (2|k|-1)^3=8|k|^3-12|k|^2+6|k|-1$. This implies that $-2 (4 |k|^3 - 10 |k|^2 + |k| - 2) > 0$, which is impossible for $|k| > 2.5$.
\end{proof}

By solving the reccurence relations defining $(V_n)$ and $(W_m)$, we get that
\begin{align*} V_n &= \frac{\sqrt{k+1}+\sqrt{k-1}}{2\sqrt{k+1}}(k+\sqrt{k^2-1})^n + \frac{\sqrt{k+1}-\sqrt{k-1}}{2\sqrt{k+1}}(k-\sqrt{k^2-1})^n,\\
W_m &= \frac{1}{2\sqrt{16k^3-4k}} \big( (x_1\sqrt{16k^3-4k}+z_1 \sqrt{k-1})(4k^2-2k-1+\sqrt{(16k^3-4k)(k-1)})^m  +\\
& \, + (x_1\sqrt{16k^3-4k}-z_1 \sqrt{k-1})(4k^2-2k-1-\sqrt{(16k^3-4k)(k-1)})^m \big).
\end{align*}

Let $\ds P'=\frac{1}{\sqrt{c}}(x_1\sqrt{c}+z_1 \sqrt{a})(s+\sqrt{ac})^m$  i $\ds Q'= \frac{1}{\sqrt{b}}(\sqrt{a}+\sqrt{b})(r+\sqrt{ab})^n$. %\frac{1}{\sqrt{k+1}}(\sqrt{k-1}+\sqrt{k+1})(k+\sqrt{k^2-1})^n$.
We remark that $Q' \neq Q$ and $P'=\sqrt{\frac{a}{c}}P$. However, with $m\geqslant 3$, we have the same bounds on $Q'$ and $|P'|-|Q'|$. They are obtained in a similar manner:
\begin{align*} |Q'| &=\frac{1}{\sqrt{|b|}}|\sqrt{a}+\sqrt{b}|\cdot|r+\sqrt{ab}|^n \geqslant \frac{|b-a|}{\sqrt{|b|}}\cdot\frac{1}{|\sqrt{b}-\sqrt{a}|}|\sqrt{ab}|^3 \\
&\geqslant \frac{2}{\sqrt{|b|}}\cdot\frac{1}{|\sqrt{b}|+|\sqrt{a}|}|ab|^{3/2}
\geqslant 12\frac{|b|}{|a|},
\end{align*}
since $|a|^{5/2} \geqslant 6(\sqrt{|b|}+\sqrt{|a|})$, i.~e.~$|k+1|^{5/2} \geqslant 6(\sqrt{|k+1|}+\sqrt{|k-1|})$, which holds for $|k|\geqslant 4.846$. If $|k|\geqslant 23$, this implies that $|Q'|\geqslant 11$ since $12\frac{|k+1|}{|k-1|}\geqslant 12 \frac{|k|-1}{|k|+1} \geqslant 11$, which is equivalent to $|k|\geqslant 23$. Similarly, it holds that $|P'|\geqslant 12$ so
\begin{align*}
||P'|-|Q'|| &\leqslant \left| \frac{c-a}{c}(P')^{-1}-\frac{b-a}{b}(Q')^{-1}\right| \leqslant \left|1-\frac{a}{c}\right|\frac{1}{|P'|}+\left|1-\frac{a}{b}\right|\frac{1}{|Q'|} \\
&\leqslant \frac{1}{12}\left|1-\frac{a}{c}\right|+\frac{1}{11}\frac{2}{|b|} < \frac{5}{48}+\frac{5}{48} = \frac{5}{24},
\end{align*}

\noindent for $|b|\geqslant |k|-1 > \frac{96}{55}$. Therefore, the conclusion of  Lemma \ref{lforma} holds for the linear form $\Gamma = \log\Lambda'=\log\frac{|P'|}{|Q'|}$ as well.  %\ref{lforma}.

\subsection{Minimal polynomials}
Let $k=\mu+i\nu$ and
\begin{align*}
\alpha_1 &= |k+\sqrt{k^2-1}|, \\
\alpha_2 &= |4k^2-2k-1+\sqrt{(16k^3-4k)(k-1)}| \text{    i} \\
\alpha_3 &= \left|\frac{\sqrt{16k^3-4k}(\sqrt{k-1}+\sqrt{k+1})}{\sqrt{k+1}(x_1\sqrt{16k^3-4k}+z_1\sqrt{k+1})}\right|.
\end{align*} 

The minimal polynomial for $\alpha_1$ is $p_1(x)=x^8-4(\mu^2+\nu^2)x^6+(8\mu^2-8\nu^2-2)x^4-4(\mu^2+\nu^2)x^2+1$, according to \cite{Zrinka}. In the same paper, it was shown that $h(\alpha_1) \leqslant \frac 14\log{(2|k|+1)}$.

The minimal polynomial for $\alpha_2$ is determined with the help of Mathematica \cite{w:m11}, 
\begin{align*}
 p_2(x) =& \, x^8-4\left((16(\mu^2+\nu^2)-16\mu-4)(\mu^2+\nu^2)+16\nu^2+4\mu+1\right)x^6+\\
 & \, + \left((128\mu^4-128\mu^3-32\mu^2+32\mu+6)+(-768\mu^2+384\mu+32)\nu^2+128\nu^4\right)x^4\\
 & \, - 4\left((16(\mu^2+\nu^2)-16\mu-4)(\mu^2+\nu^2)+16\nu^2+4\nu+1)\right)x^2+1.
\end{align*}

Polynomial $p_2(x)$ has the following zeroes:
\begin{align*} 
x_{1,2} &= \pm\alpha_2, &x_{5,6} = \pm\sqrt{|s|^2 - |ac| + \sqrt{(|s|^2 - |ac|)^2 - 1}},\\
x_{3,4} &= \pm |s - \sqrt{ac}|, &x_{7,8} = \pm\sqrt{|s|^2 - |ac| - \sqrt{(|s|^2 - |ac|)^2 - 1}},
\end{align*}
and $|x_i|=1$ for $i=5, 6, 7,  8$. It follows that \[ h(\alpha_2) \leqslant \frac 18 \log{|x_1||x_2|} = \frac 14 \log|4k^2-2k-1+\sqrt{(16k^3-4k)(k-1)}|, \]
which implies $h(\alpha_2) \leqslant \frac 14 \log|9k^2| = \frac 12 \log 3|k|$.
%\begin{comment}
Polynomial $p_2(x)$ has the following zeroes $x_{1,2} = \pm\alpha_2$, $\ds x_{3,4} = \pm |s - \sqrt{ac}|$, $\ds x_{5,6} = \pm\sqrt{|s|^2 - |ac| + \sqrt{(|s|^2 - |ac|)^2 - 1}}$ and $x_{7,8} = \pm\sqrt{|s|^2 - |ac| - \sqrt{(|s|^2 - |ac|)^2 - 1}}$,
and $|x_i|=1$ for $i=5, 6, 7,  8$. This implies that \[ h(\alpha_2) \leqslant \frac 18 \log{|x_1||x_2|} = \frac 14 \log|4k^2-2k-1+\sqrt{(16k^3-4k)(k-1)}|, \]
and, consequently $h(\alpha_2) \leqslant \frac 14 \log|9k^2| = \frac 12 \log 3|k|$.
%\end{comment}

\subsection{Bounding the conjugates of $\alpha_3$}
\begin{lemma}\label{lem:conj} If $|k|\geqslant 10^7$, then, for all conjugates $\alpha_3'$ of $\alpha_3$, it holds that $|\alpha_3'|\leqslant |k|^4$.
\end{lemma}
\begin{proof}
One can guess the minimal polynomial for $\alpha_3$ and all conjugates. The first eight are $x_{1,2}' = \pm \alpha_3$, $\ds x_{3, 4}= \pm \left| \frac{\sqrt{c}(\sqrt{b}+\sqrt{a})}{\sqrt{b}(x_1\sqrt{c}-z_1\sqrt{a})}\right|$, $\ds
x_{5,6}= \pm \left| \frac{\sqrt{c}(\sqrt{b}-\sqrt{a})}{\sqrt{b}(x_1\sqrt{c}-z_1\sqrt{a})}\right|$, $\ds x_{7,8}= \pm \left| \frac{\sqrt{c}(\sqrt{b}-\sqrt{a})}{\sqrt{b}(x_1\sqrt{c}-z_1\sqrt{a})} \right|$.
\vskip 0.4em
Furthermore, $x_9, \dotsc, x_{12}$ are zeroes of
\[ q_1(x) = x^4-2\left|\frac{c}{b(x_1\sqrt{c}+z_1\sqrt{a})^2}\right|(|b|-|a|)x^2+\left|\frac{c(b-a)}{b(x_1\sqrt{c}+z_1\sqrt{a})^2}\right|^2, \]
$x_{13},\dots, x_{16}$ of
\[ q_2(x) = x^4-2\left|\frac{c}{b(x_1\sqrt{c}-z_1\sqrt{a})^2}\right|(|b|-|a|)x^2+\left|\frac{c(b-a)}{b(x_1\sqrt{c}-z_1\sqrt{a})^2}\right|^2, \]
the next eight of
\[ q_3(x) = x^4-2\left|\frac{c(\sqrt{b}+\sqrt{a})^2}{b(c-a)^2}\right|(|cx_1^2|-|az_1^2|)x^2+\left|\frac{c(\sqrt{b}-\sqrt{a})^2}{b(c-a)}\right|^2 \quad \text{  and}\]
\[ q_4(x)  = x^4-2\left|\frac{c(\sqrt{b}-\sqrt{a})^2}{b(c-a)^2}\right|(|cx_1^2|-|az_1^2|)x^2+\left|\frac{c(\sqrt{b}-\sqrt{a})^2}{b(c-a)}\right|^2, \]
then
\[ q_5(x)= x^4-2\left|\frac{c}{b(c-a)^2}\right|(|x_1\sqrt{bc}+z_1a|^2-|x_1\sqrt{ac}+z_1\sqrt{ab}|^2)x^2+\left|\frac{c(b-a)}{b(c-a)}\right|^2 \]
and finally
\[ q_6(x) = x^4-2\left|\frac{c}{b(c-a)^2}\right|(|x_1\sqrt{bc}-z_1a|^2-|x_1\sqrt{ac}-z_1\sqrt{ab}|^2)x^2+\left|\frac{c(b-a)}{b(c-a)}\right|^2 . \]
\vskip -0.5em

This suffices to find the bound we need here. Namely, the zero of the monic polynomial is bounded from above by the sum of the absolute values of its coefficients. For this polynomial, we can see that the coefficients have at most the order of $|c|^2 \cdot |a|$ (or $\cdot |b|$). More precisely, we will show that all the coefficients of $x^2$-terms are less than $3|k|^7$, while all the free coefficients are less than $1025|k|^7$ for $k$ large enough.

The coefficient of $x^2$ in $q_1$ and $q_2$ is less than or equal to $2|c|(|b|+|a|)\leqslant 2|16k^3-4k|(2|k|+2)\leqslant |k|^5$ for $|k|\geqslant 65$. This type of claim is proven as earlier in the paper, by using the triangle inequality and analysing the obtained functions of $|k|$. The coefficient of $x^2$ in $q_3$ and $q_4$ is less than or equal to 
\[ \frac{2|c|(\sqrt{|a|}+\sqrt{|b|})(|cx_1^2|+|az_1^2|)}{|b|^2|c-a|} \leqslant \frac{2|16k^3-4k|\cdot2\sqrt{|k|+1}(150|k|^5+65)}{(|k|-1)^2(16|k|^3-5|k|-1)} <  |k|^4 \, \text{        for } |k|\geqslant 1.45\cdot10^6 .\]

\noindent Similarly, the coefficient of $x^2$ in $q_5$ and $q_6$ is less than $3|k|^7$ for $|k|\geqslant 2$.

The free coefficients are less than $|16k^3-4k|^2 (2\sqrt{|k|+1})^2 \leqslant 1025|k|^7$ for $|k|\geqslant 1025$.

Therefore, $|\alpha_3'|^2\leqslant 1028|k|^7$ for every conjugate $\alpha_3'$, i.~e.~$|\alpha_3'| \leqslant |k|^4$ for $|k|\geqslant 10^7$.

\end{proof}
\subsection{The final result}
%\newpage
\quad \\ We denote Mahler measure as $M(\alpha)$ and logarithmic Weil's height as $h(\alpha)$.
\begin{lemma}\label{lem:bw}
If $|k|\geqslant 5\cdot 10^{37}$, $\Gamma\neq 0$ and $V_n=W_m$, then $m\leqslant 2$ or $n\leqslant 2$.
\end{lemma}
\begin{proof}
Assume that, on the contrary, $V_n=W_m$ and $n\geqslant m\geqslant 3$. It holds that $\ds M(\alpha_3) \leqslant |a_d|\prod_{i=1}^d \max\{|\alpha^{'}|, 1\}$ and $\ds |a_d| \leqslant \left(\sqrt{|k|+1}(|x_1|\sqrt{16|k|^3+4|k|}+|z_1|\sqrt{|k|+1})\right)^{32} < 257^{16} |k|^{65}$. Since Lemma \ref{lem:conj} provides the bound for conjugates  $|\alpha_3'| \leqslant |k|^4$, it follows that \[ h(\alpha_3) \leqslant \frac{1}{32}\log{(257^{16}|k|^{65}\cdot |k|^{4\cdot 32})} = 2.774538+\frac{193}{32}\log |k|.\]

For all three bounds it holds that $h'(\alpha_i) \leqslant 7\log|k|$.
%\ref{lforma}

Lemma \ref{lforma} implies $|\Gamma|=|\log \Lambda'| < K\sqrt{|ac|}|s+\sqrt{ac}|^{-m}$ (if $m, n\geqslant 3$), where $K=\frac 83 \log{\frac{24}{19}} =  0.622973$. Since $|s+\sqrt{ac}|\geqslant \sqrt{|ac|}=\sqrt{|(16k^3-4k)(k-1)|}$, it follows that $|s+\sqrt{ac}| > 3|k|^2$ (for $|k|>3$). Hence
\[ |\Gamma| < K\sqrt{|ac|}(3|k|^2)^{-m} < K(1.5|k|)^{1-m}. \]
%(Ovo je dosta gruba ograda.)

%Budući da $\frac{|16k^3-4k|}{|k-1|}$ i $\frac{|k+1|}{|k-1|}$ nisu racionalni, po Lemi \ref{lema:fneq0}
We now apply the following well-known theorem from \cite{bakerw}.

\begin{theorem}[{Baker, W\"ustholz}]
Let $\Gamma = b_1\log \alpha_1 + b_2\alpha_2 + \dots+b_n\log\alpha_n$ be a linear form in logarithms of algebraic numbers $\alpha_1, \alpha_2, \dotsc, \alpha_n$ with integer coefficients $b_1, b_2, \dots, b_n$. If $\Gamma \neq 0$, then
\[ \log |\Gamma| \geqslant -18(n+1)!n^{n+1} (32d)^{n+2} \log(2nd) h'(\alpha_1)h'(\alpha_2)\dots h'(\alpha_n) \log B,\]
where $d=[\Q(\alpha_1, \alpha_2, \dotsc, \alpha_n):\Q], B=\max\{|b_1|, |b_2|, \dotsc, |b_n|\}$, and $h'(\alpha)=\max\{h(\alpha),\frac{1}{d}|\log \alpha|,\frac{1}{d}\}$.
\end{theorem}
%WEIL'S !!!

We use the logarithm only for absolute values (positive reals), but this theorem holds more generally. For the logarithm of a complex number $z=re^{i\varphi}$ with $r>0$ we can take $\log z = \log r + i\varphi$.

Baker--W\"ustholz theorem implies that, if $\Gamma\neq 0$, then
\[ \log K(1.5|k|)^{1-m} > \log|\Gamma| >  -18\cdot 4!\cdot 3^4(32\cdot 2048)^5 \cdot 343\log^3|k|\log{(6\cdot 2048)}\log n,\]
meaning $(1-m)\log{\frac 32|k|} > \log K+(1-m)\log{\frac 32|k|} > -K'\log^3|k|\log n$, where $\log K\approx -0.47325$. Together with Lemma \ref{lema:vmn}, this implies that
\[ \frac{m-1}{\log{(3m+2)}} \leqslant \frac{m-1}{\log n} < K' \frac{\log^3|k|}{\log \frac 32 |k|} < K'\log^2 |k|,\]
where $K'=18\cdot 4!\cdot 3^4(32\cdot 2048)^5 \cdot 343\log{(6\cdot 2048)} \approx 1.3663\cdot 10^{32}$.

Since $m\geqslant 2|k|-1$ (by \eqref{eq:mk}) and since function $f(x)=\frac{x-1}{\log{(3x+2)}}$ is increasing, it follows that
$\ds \frac{2|k|-2}{\log{(6|k|-1)}} < K'\log^2|k|$ and $|k|-1 < 6.831506\cdot 10^{31} \log^2|k|\log{(6|k|-1)}$, which is impossible for $|k|\geqslant 5\cdot 10^{37}$.
\end{proof}

\begin{lemma}\label{bkroza} If $k$ is a Gaussian integer such that $\Imz k \Rez k\neq 0$, then $\ds \frac{|k+1|}{|k-1|}\not\in\Q$.
\end{lemma}
\begin{proof} It is sufficient to show that $|k+1|\cdot|k-1|$ is not in $\Q$, so it also suffices to show that it is not a rational integer. If $k=x+yi$, then $|k+1|^2\cdot|k-1|^2=x^4+y^4+1+2x^2y^2-2x^2+2y^2$, so we need to show that this expression is not a perfect square.

It holds that $x^4+y^4+1+2x^2y^2-2x^2+2y^2 > (x^2+y^2-2)^2$, because this is equivalent to $2 x^2 + 6 y^2 > 3$. Similarly, $x\neq 0$ implies that $x^4+y^4+1+2x^2y^2-2x^2+2y^2 < (x^2+y^2+1)^2$.

From $x^4+y^4+1+2x^2y^2-2x^2+2y^2=(x^2+y^2)^2$, it follows that $1 - 2 x^2 + 2 y^2=0$, which is impossible (parity check), while $x^4+y^4+1+2x^2y^2-2x^2+2y^2=(x^2+y^2-1)^2$ implies that $4y^2=0$, again impossible since $y\neq 0$.
\end{proof}

\begin{lemma}\label{ckroza} If $k$  is a Gaussian integer such that $\Imz k \neq 0$, then $\ds \frac{|16k^3-4k|}{|k-1|} \not\in\Q$
\end{lemma}

\begin{proof} Again, it suffices to show that $|16k^3-4k|\cdot|k-1| \not\in\Z$, i.~e.~$|4k^3-k|^2\cdot|k-1|^2$ is not a perfect square.

If $k=x+yi$, then $|4k^3-k|^2\cdot|k-1|^2 =| 4 x^4 - 4 x^3 - 24 x^2 y^2 - x^2 + 12 x y^2 + x + 4 y^4 + y^2+ i (16 x^3 y - 12 x^2 y - 16 x y^3 - 2 x y + 4 y^3 + y)|^2=(x^2-2x+1+y^2)(4x^2-4x+1+4y^2)(4x^2+4x+1+4y^2)(x^2+y^2)$.
By substituting $z=-4x+1$, we get $(z^4+(32y^2-10)z^2+160y^2+256y^4+9)^2 + 4096y^2z^2$. Further substitutions $u=z^2, v=y^2$ give that $(u^2+(32v-10)u+160v+256v^2+9)^2 + 4096uv$ is a square, where $u$ and $v$ are also perfect squares. Since $y=\Imz k\neq 0$, we see that $v \neq 0$, and neither $u=(-4x+1)^2\neq 0$. Hence $(u^2+(32v-10)u+160v+256v^2+9)^2 + 4096uv > (u^2+(32v-10)u+160v+256v^2+9)^2$. If the left-hand side is a square, there is a positive integer $w$ such that $(u^2+(32v-10)u+160v+256v^2+9)^2 + 4096uv=(u^2+(32v-10)u+160v+256v^2+9+w)^2$.

The equation $2w(u^2+(32v-10)u+160v+256v^2+9)+w^2-4096uv=0$ is quadratic in $u$. The discriminant is
\begin{align*}
4D(v, w) &= -8( w^3 - 32 w^2 + 65536 v^2 w -2097152 v^2 + 640 v w^2  - 20480 v w)\\
&= -4\cdot 2((w^2 (w - 32) + 65536 v^2 (w - 32) + 640 wv (w - 32)),
\end{align*}
which is negative for $w > 32$. For a solution to be an integer, the discriminant must be a perfect square. It follows that $w\in\{1, 2, \dotsc, 32\}$.

Since $D(v, 32)=0$, solving the quadratic equation implies that $u=16v+5$, which is not a square because $5$ is not a quadratic remainder modulo $16$. For the most of the remaining values of $w$, in a similar manner we show that $D(v, w)$ is not a square. Observe that $D(v, w)\equiv -2w^2(w-32)\pmod{128}$.

For odd $w$, we get that $D(v, w)\equiv 2\pmod{4}$,  which cannot be a square. For $w\equiv 2\pmod{8}$, it holds that $D(v, w)\equiv -16\pmod{64}$, while for $w\equiv 4\pmod{8}$, $D(v, w)\equiv 128\pmod{256}$ and again the discriminant cannot be a square. %($48$ i $128$ nisu kvadrati modulo $256$).

For $w=6$, $\frac{D(v, 6)}{16}=212992 v^2 + 12480 v + 117 \equiv 5\pmod{16}$, so $D(v, 6)$ is not a square. Similarly, $\frac{D(v, 8)}{256} \equiv 12\pmod{16}$, $\frac{D(v, 16)}{4096}\equiv 2 \pmod{4}$ and $\frac{D(v, 22)}{16} \equiv 13\pmod{16}$ imply that none of $D(v, 8)$, $D(v, 16)$ and $D(v, 22)$ can be a square.

For $w=14$, it holds that $(128v+9)^2 > \frac{D(v, 14)}{144} =(128v+7)^2+448v$, hence $D(v, 14)$ is not a square since $v>0$. Similarly, $(32v+4)^2>\frac{D(v, 24)}{1024}=1024v^2+240v+9=(32v+3)^2+48$ and $(128v+19)^2 > \frac{D(v, 30)}{16}=(128v+15)^2+960v$ show that neither $D(v, 24)$ nor $D(v, 30)$ is a square ($v>0$).

We have checked all $w\in\{1, 2, \dotsc, 31, 32\}$ and thus proven the lemma.
\end{proof}

\begin{theorem}
Let $k$ be a Gaussian integer such that $\Rez k \neq 0$ and $|k|\geqslant 5\cdot 10^{37}$. The Diophantine triple $\{k-1, k+1, 16k^3-4k\}$ can be extended to a Diophantine quadruple only by $d=4k$ or $d=64k^5 -48k^3+8k$. 
\end{theorem}

\begin{proof}
If $\Imz k = 0$, then the elements of the sequence $(V_n)_n$ are integers, hence $x$ is an integer too. Since $(k-1)d+1=x^2$, it follows that $d\in\Q\cap\Z[i]$, so $d$ is an integer. The sign of $d$ is the same as the sign of $k-1, k+1$ and $16k^3-4k$ (because $d=\frac{x^2-1}{k-1} \neq 0$), so Theorem 1 from \cite{16k3z}, since $|k|\geqslant 2$, implies that $d=4k$ or $d=64k^5 -48k^3+8k$.

From now on, we assume that $\{a, b, c, d\}$ is a Diophantine quadruple for $a=k-1, b=k+1, c=16k^3-4k$ and that $\Imz k$ is not $0$.

%$V_3=$
Checking $V_n$ and $W_m$ for small indices, we obtain the extensions $4k$, $64k^5-48k^3+8k$ and candidates such as $W_1^{(1)}=4k^2-k-2$. By computing the first few elements of $(V_n)_n$, $V_1=2k-1$ and $V_2=4k^2-2k-1$, we see that $W_1^{(2)}$ cannot have the same value (for large $|k|$), and neither can the larger elements $V_n$ for $n\geqslant 3$, since these are greater  in absolute value than $W_1^{(1)}$:
$|V_n|-|W_1^{(1)}|\geqslant |V_3-W_1^{(1)}|=|8k^3-4k^2-4k+1-(4k^2-k-2)|>0$ for $|k|>10^{37}$. Similarly, $V_1=2k-1$ and $V_2=4k^2-2k-1$ cannot be an element of the sequence $W_m^{(1)}$. Analogously we check sequences $W_m^{(j)}$ for the remaining $j=2, 3, 4, 5, 6$.

Therefore, indices $n$ and $m$ are greater than $2$ if $d\not\in \{4k, 64k^5-48k^3+8k\}$.

Lemma \ref{bkroza} and Lemma \ref{ckroza} imply that $\frac{|c|}{|a|}$ and $\frac{|b|}{|a|}$ are not rational numbers. In the same manner as in the Lemma \ref{lema:fneq0}, this implies that the linear form $\Gamma=\log\Lambda'$ is not $0$. If $V_n=W_m$ for $m\geqslant 2$ and $n\geqslant 2$, then Lemma \ref{lem:bw} would imply that $|k| < 5\cdot 10^{37}$, which is a contradiction. Therefore, the assumption that $V_n=W_m$ for $m\geqslant 3$ and $n\geqslant 3$ is wrong, and so is the claim that $d\not\in \{4k, 64k^5-48k^3+8k\}$ for $|k| \geqslant 5\cdot10^{37}$.
\end{proof}
\begin{comment}
\end{comment}

\bibliographystyle{amsplain}
%\bibliography{bibliography}
\providecommand{\bysame}{\leavevmode\hbox to3em{\hrulefill}\thinspace}
\providecommand{\MR}{\relax\ifhmode\unskip\space\fi MR }
\providecommand{\MRhref}[2]{%
  \href{http://www.ams.org/mathscinet-getitem?mr=#1}{#2}
}
\providecommand{\href}[2]{#2}

%\bibliography{bibliography}

%\begin{comment}

\section*{Acknowledgements}
N.~A.~and A.~F.~were supported by the Croatian Science Foundation under the project no.~IP-2018-01-1313.

%\end{comment}

\end{document}